\documentclass[twoside]{article}

\usepackage[accepted]{aistats2022}
\usepackage[round]{natbib}

\usepackage{amssymb, amsmath, amsthm, latexsym}
\usepackage{url}
\usepackage{algorithm}
\usepackage{algorithmic}
\usepackage{tabularx}
\usepackage{paralist}
\usepackage{mathtools}

\usepackage{bbm} 

\usepackage{makecell}
\usepackage{multirow}
\usepackage{booktabs}

\usepackage{nicefrac}       

\usepackage[flushleft]{threeparttable} 

\usepackage{caption}
\usepackage{multirow}
\usepackage{colortbl}
\definecolor{bgcolor}{rgb}{0.8,1,1}
\definecolor{bgcolor2}{rgb}{0.8,1,0.8}
\definecolor{niceblue}{rgb}{0.0,0.19,0.56}

\usepackage{hyperref}
\hypersetup{colorlinks,linkcolor={blue},citecolor={niceblue},urlcolor={blue}}

\usepackage[dvipsnames]{xcolor}
\usepackage{pifont}

\usepackage{subfigure}

\newcommand{\R}{\mathbb{R}}

\def\<#1,#2>{\left\langle #1,#2\right\rangle}

\newtheorem{lemma}{Lemma}[section]
\newtheorem{theorem}{Theorem}[section]
\newtheorem{definition}{Definition}[section]

\theoremstyle{plain}

\usepackage{xspace}

\newcommand{\algname}[1]{{\sf  #1}\xspace}

\usepackage[colorinlistoftodos,bordercolor=orange,backgroundcolor=orange!20,linecolor=orange,textsize=scriptsize]{todonotes}

\renewcommand{\Re}{\mathrm{Re}}
\renewcommand{\Im}{\mathrm{Im}}
\newcommand{\Sp}{\mathrm{Sp}}
\newcommand{\Id}{\mathrm{Id}}

\newcommand{\Tr}{\mathrm{Tr}}

\newcommand{\cD}{{\cal D}}

\newcommand{\cH}{{\cal H}}

\newcommand{\cO}{{\cal O}}

\newcommand{\mA}{{\bf A}}

\newcommand{\mG}{{\bf G}}

\newcommand{\mI}{{\bf I}}

\newcommand{\mM}{{\bf M}}

\newcommand{\mU}{{\bf U}}

\newcommand{\EE}{\mathbb{E}}
\newcommand{\CC}{\mathbb{C}}

\newcommand{\tx}{\widetilde{x}}
\newcommand{\ty}{\widetilde{y}}
\newcommand{\hx}{\widehat{x}}
\newcommand{\hy}{\widehat{y}}

\usepackage{hyperref}
\graphicspath{{../plots/}}

\usepackage{makecell}

\usepackage{accents}
\newlength{\dhatheight}

\usepackage{pgfplotstable} 
\usetikzlibrary{automata, positioning, arrows, shapes, fit, calc, intersections}
\usepgfplotslibrary{statistics}
\include{figure}

\begin{document}

%
\runningtitle{Extragradient Method: $\cO\left(\nicefrac{1}{K}\right)$ Last-Iterate Convergence for Monotone Variational Inequalities}

%

\twocolumn[

\aistatstitle{Extragradient Method: $\cO\left(\nicefrac{1}{K}\right)$ Last-Iterate Convergence for Monotone Variational Inequalities\\ and Connections With Cocoercivity}

\aistatsauthor{ Eduard Gorbunov \And Nicolas Loizou \And Gauthier Gidel}

\aistatsaddress{ \begin{tabular}{c}
    MIPT, Russia\\ Mila \& UdeM, Canada
\end{tabular} \And  \begin{tabular}{c}
    Johns Hopkins University \\
Baltimore, USA
\end{tabular} \And \begin{tabular}{c}
    Mila \& UdeM, Canada\\  Canada CIFAR AI Chair
\end{tabular}} ]

\begin{abstract}
    Extragradient method (\algname{EG}) \citep{korpelevich1976extragradient} is one of the most popular methods for solving saddle point and variational inequalities problems (VIP). Despite its long history and significant attention in the optimization community, there remain important open questions about convergence of \algname{EG}. In this paper, we resolve one of such questions and derive the first \textit{last-iterate} $\cO(\nicefrac{1}{K})$ convergence rate for \algname{EG} for monotone and Lipschitz VIP \textit{without any additional assumptions} on the operator unlike the only known result of this type \citep{golowich2020last} that relies on the Lipschitzness of the Jacobian of the operator. The rate is given in terms of reducing the squared norm of the operator. Moreover, we establish several results on the (non-)cocoercivity of the update operators of \algname{EG}, Optimistic Gradient Method, and Hamiltonian Gradient Method, when the original operator is monotone and Lipschitz. 
\end{abstract}

\section{INTRODUCTION}
Saddle point problems receive a lot of attention during recent years, especially in the machine learning community. These problems appear in various applications such as robust optimization \citep{ben2009robust} and control \citep{hast2013pid}, adversarial training \citep{goodfellow2015explaining, madry2018towards} and generative adversarial networks (GANs) \citep{goodfellow2014generative}. Saddle point problems are often studied from the perspective of variational inequality problems (VIP) \citep{harker1990finite,ryu2020large, gidel2019variational}. In the unconstrained case, VIP is defined as follows:
\begin{equation}
    \text{find } x^* \in \R^d \quad \text{such that}\quad F(x^*) = 0, \label{eq:main_problem} \tag{VIP}
\end{equation}
where $F:\R^d \to \R^d$ is some operator.

Such problems are usually solved via first-order methods due to their practical efficiency. The simplest example of such a method is Gradient Descent (\algname{GD}): $x^{k+1} = x^k - \gamma F(x^k)$. However, there exist examples of simple (monotone and $L$-Lipschitz) problems such that \algname{GD} does not converge to the solution. To circumvent this issue Extragradient Method (\algname{EG}) $x^{k+1} = x^k - \gamma F(x^k - \gamma F(x^k))$ \citep{korpelevich1976extragradient} and Optimistic Gradient Method (\algname{OG}) $x^{k+1} = x^k - 2\gamma F(x^k) + \gamma F(x^{k-1})$ \citep{popov1980modification} were proposed. After their discovery, these methods were revisited and extended in various ways, e.g., stochastic \citep{gidel2019variational, mishchenko2020revisiting, hsieh2020explore,li2021convergence}, distributed \citep{liu2020decentralized, beznosikov2020distributed, beznosikov2021decentralized}, and non-Euclidean versions \citep{juditsky2011solving, azizian2021last} were proposed and analyzed.

Surprisingly, despite the long history of and huge interest in \algname{EG} and \algname{OG}, there exist significant gaps in the theory of these methods. In particular, it is well known that both methods converge in terms of $\min_{k=0,1,\ldots,K}\|F(x^k)\|^2$ with rate $\cO(\nicefrac{1}{K})$ for monotone $L$-Lipschitz operator $F$ \citep{solodov1999hybrid, ryu2019ode}. Although such \textit{best-iterate} guarantees provide valuable information about the rate of convergence, they do not state anything about \textit{last-iterate} convergence rate. Recently, this limitation was partially resolved in \citet{golowich2020last,golowich2020tight} where the authors proved last-iterate $\cO(\nicefrac{1}{K})$ convergence rate for \algname{EG} and \algname{OG} \textit{under the additional assumption} that the Jacobian $\nabla F(x)$ of operator $F(x)$ is $\Lambda$-Lipschitz. However, the obtained rates depend on the $\Lambda$ that can be much larger than $L$ or even undefined for some operators (see Appendix~\ref{sec:details_on_lip_jac}). That is, the following important question remains open:
\begin{gather*}
    \text{\color{blue}Q1: }\text{\it Is it possible to prove last-iterate $\cO(\nicefrac{1}{K})$}\\ \text{\it convergence rate for \algname{EG}/\algname{OG} when $F$ is monotone}\\ \text{\it and $L$-Lipschitz \underline{without additional assumptions}?}
\end{gather*}

Next, there is a noticeable activity in the analysis of various methods for solving \eqref{eq:main_problem} under the cocoercivity assumption on $F$ during the last few years \citep{chavdarova2019reducing,malinovskiy2020local,loizou2021stochastic}. Unfortunately, this assumption is stronger than monotonicity and Lipschitzness of $F$: it does not hold even for bilinear games. However, under the cocoercivity of $F$ the analysis of some methods becomes extremely simple. For example, if operator $F$ is cocoercive, then one can easily prove last-iterate $\cO(\nicefrac{1}{K})$ convergence rate for \algname{GD} \citep{brezis1978produits, diakonikolas2021potential}.

Furthermore, it is known that Proximal Point operator $F_{\algname{PP},\gamma}(x)$ implicitly defined as $F_{\algname{PP},\gamma}(x) = F(y)$, where $y = x - \gamma F(y)$, is cocoercive for any monotone $F$ (Corollary 23.10 from \citep{bauschke2011convex}). Therefore, Proximal Point method (\algname{PP}) $x^{k+1} = x^k - \gamma F(x^{k+1})$ \citep{martinet1970regularisation, rockafellar1976monotone} can be seen as \algname{GD} for operator $F_{\algname{PP},\gamma}$ and last-iterate $\cO(\nicefrac{1}{K})$ convergence rate follows from the analysis of \algname{GD} under the cocoercivity. Since \algname{EG} and \algname{OG} are often considered as approximations of \algname{PP} when $F$ is $L$-Lipschitz \citep{mokhtari2019proximal}, there is a hope that \algname{EG} and \algname{OG} can be rewritten as \algname{GD} for some cocoercive operator. In particular, for \algname{EG} one can consider $F_{\algname{EG},\gamma}(x) = F(x - \gamma F(x))$ and get that \algname{EG} for $F$ is \algname{GD} for $F_{\algname{EG},\gamma}$. Using matrix notation and rewriting \algname{OG} using $z^k = ((x^{k})^\top, (x^{k-1})^\top)^\top$, one can also construct $F_{\algname{OG},\gamma}(x)$ and consider \algname{OG} as \algname{GD} for this operator. Keeping in mind the simplicity of getting last-iterate $\cO(\nicefrac{1}{K})$ convergence rate for \algname{GD} under the cocoercivity, the following question arises:
\begin{gather*}
    \text{\color{blue}Q2: }\text{\it Are operators $F_{\algname{EG},\gamma}$ and $F_{\algname{OG},\gamma}$ cocoercive}\\ \text{\it when $F$ is monotone and $L$-Lipschitz?}
\end{gather*}

In this paper, we give a positive answer to the first question ({\color{blue}Q1}) and negative answer to the second question ({\color{blue}Q2}). Before we summarize our main contributions, we introduce necessary definitions.

\subsection{Preliminaries}
If the opposite is not specified, throughout the paper we assume that operator $F$ from \eqref{eq:main_problem} is monotone
\begin{equation}
    \langle F(x) - F(x'), x - x'\rangle \ge 0 \quad \forall x,x' \in \R^d, \label{eq:monotonicity_def}
\end{equation}
and $L$-Lipschitz
\begin{equation}
    \|F(x) - F(x')\| \le L\|x - x'\|\quad  \forall x,x' \in \R^d. \label{eq:L_lip_def}
\end{equation}

Next, we also rely on the definition of cocoercivity.
\begin{definition}[Cocoercivity]
    Operator $F: \R^d \to \R^d$ is called $\ell$-cocoercive if for all $x,x' \in \R^d$
    \begin{equation}
        \|F(x) - F(x')\|^2 \le \ell \langle F(x) - F(x'), x - x' \rangle. \label{eq:l_cocoercivity}
    \end{equation}
\end{definition}

Using Cauchy-Schwarz inequality, one can easily show that $L$-cocoercivity of $F$ implies its monotonicity and $L$-Lipschitzness. The opposite is not true: it is sufficient to take $F$ corresponding to the bilinear game (see \citep{carmon2019variance} and references therein).

\paragraph{Measures of convergence.} In the literature on VIP, the convergence of different methods is often measured via so-called merit or gap functions, e.g., restricted gap function $\texttt{Gap}_F(x) = \max_{y\in\R^d:\|y-x^*\| \le R}\langle F(y), x - y \rangle$, where $R \sim \|x^0 - x^*\|$ \citep{nesterov2007dual}. When $F$ is monotone, $\texttt{Gap}_F(x)$ can be seen as a natural extension of optimization error for VIP. However, it is unclear how to tightly estimate $\texttt{Gap}_F(x)$ in practice and how to generalize it to non-monotone case. From this perspective, the squared norm of the operator is preferable as a measure of convergence (see \citep{yoon2021accelerated} and references therein). Therefore, we focus on $\|F(x^K)\|^2$. We notice here that in the constrained case (squared) norm of the operator is not a valid measure of convergence.

\subsection{Contributions}
Below we summarize our main contributions.

\begin{itemize}
    \item We prove that $\|F(x^K)\|^2 = \cO\left(\nicefrac{1}{K}\right)$ where $x^K$ is generated after $K$ iterations of Extragradient Method (\algname{EG}) applied to solve \eqref{eq:main_problem} with monotone $L$-Lipschitz operator $F$ (Theorem~\ref{thm:EG_last_iter_conv_non_linear}). That is, we derive the first last-iterate $\cO\left(\nicefrac{1}{K}\right)$ convergence rate for \algname{EG} under monotonicity and $L$-Lipschitzness assumptions and without additional ones. The key part of our proof is obtained via solving\footnote{Our code is available at \url{https://github.com/eduardgorbunov/extragradient_last_iterate_AISTATS_2022}. In the MATLAB code, we use PESTO \citep{taylor2017performance}, SEDUMI \citep{sedumi}, YALMIP \citep{lofberg2004yalmip} libraries, and in the part written in Python, we use CVXPY \citep{diamond2016cvxpy}.} special Performance Estimation Problem \citep{taylor2017performance, ryu2020operator}.
    \item When $F(x)$ is additionally affine, we show that $F_{\algname{EG},\gamma}$ is $\nicefrac{2}{\gamma}$-cocoercive (Lemma~\ref{lem:cocoercivity_of_EG_linear_spec_appendix}) that gives an alternative proof of last-iterate $\cO\left(\nicefrac{1}{K}\right)$ convergence for \algname{EG} based on the result for \algname{GD} (Theorem~\ref{thm:last_iter_conv_GD}).
    \item Guided by the solution of a certain Performance Estimation Problem we prove that for all $\gamma \in (0,\nicefrac{1}{L}]$ there exists $L$-Lipschitz and monotone operator $F$ such that $F_{\algname{EG},\gamma}$ is not $\ell$-cocoercive for any  $\ell > 0$ (Theorem~\ref{thm:EG_random_iter_conv}). This fact emphasizes the significant difference between \algname{EG} and Proximal Point method.
    \item We show that $F_{\algname{EG},\gamma}$ is $\nicefrac{2}{\gamma}$-star-cocoercive, i.e., cocoercive towards the solution, when $F$ is star-monotone and $L$-Lipschitz (Lemma~\ref{lem:EG_star_cocoercive}).
    \item For Optimistic Gradient method (\algname{OG}) we consider two popular representations -- standard one and extrapolation from the past (\algname{EFTP}) -- and corresponding operators $F_{\algname{OG},\gamma}$ and $F_{\algname{EFTP},\gamma}$. We prove that these operators are even non-star-cocoercive for any $\gamma > 0$ (Theorems~\ref{thm:F_OG_is_not_star_cocoercive}). This fact emphasize the difference between \algname{OG} and \algname{EG}.
    \item Finally, in the case when we additionally have Lipschitzness of the Jacobian $\nabla F$, we show that operator $F_{\cH}(x) = \nabla F(x)^\top F(x)$ of Hamiltonian Gradient Method $x^{k+1} = x^k - \gamma \nabla F(x^{k})^\top F(x^k)$ (\algname{HGM}) \citep{balduzzi2018mechanics} can be non-cocoercive when $F$ is non-affine (Theorem~\ref{thm:HGM_non_cocoercive}). Moreover, we derive best-iterate $\cO\left(\nicefrac{1}{K}\right)$ convergence rate in terms of the squared norm of the gradient of the Hamiltonian function $\cH(x) = \frac{1}{2}\|F(x)\|^2$ when $F$ and $\nabla F$ are Lipschitz-continuous but $F$ is not necessary monotone (Theorem~\ref{thm:best_iter_conv_HGM_appendix}). The details are given in Appendix~\ref{sec:HGM}.
\end{itemize}

\subsection{Related Work}

As we mention earlier, when $F$ is monotone and $L$-Lipschitz both \algname{EG} and \algname{OG} are usually analyzed in terms of the convergence for the best-iterate or the averaged-iterate. In particular, guarantees of the form $\texttt{Gap}_F(\overline{x}^K) = \cO(\nicefrac{1}{K})$ with $\overline{x}^K$ being the average of the iterates are shown in \citet{nemirovski2004prox, mokhtari2019proximal, hsieh2019convergence, monteiro2010complexity, auslender2005interior} and results like $\min_{k=0,1,\ldots,K}\|F(x^k)\|^2 = \cO(\nicefrac{1}{K})$ are given in \citet{solodov1999hybrid, ryu2019ode}. Unfortunately, these results do not provide convergence rates for the last-iterate, i.e., for $\texttt{Gap}_F(x^K)$ and $\|F(x^K)\|^2$. It turns out that both \algname{EG} and \algname{OG} satisfy the following lower bound: $\texttt{Gap}_F(x^K) = \Omega(\nicefrac{1}{\sqrt{K}})$ \citep{golowich2020last, golowich2020tight}, i.e., in terms of the gap function \algname{EG} and \algname{OG} have slower convergence for the last iterate than for the averaged iterate. 

However, as it is explained above, we focus on the convergence rates for $\|F(x^K)\|^2$. The mentioned negative results do not imply anything about the convergence in terms of $\|F(x^K)\|^2$. Moreover, for \algname{EG} and \algname{OG} \citet{golowich2020last, golowich2020tight} prove $\|F(x^K)\|^2 = \cO(\nicefrac{1}{K})$ rate under the additional assumption that the Jacobian $\nabla F(x)$ is $\Lambda$-Lipschitz. In particular, the derived rate depends on the $\Lambda$, which can be much larger than $L$ for some operators, e.g., for the gradient of logistic loss function (see Appendix~\ref{sec:details_on_lip_jac}), or simply be undefined when $\nabla F(x)$ does not exist on the whole space. In contrast, we prove $\|F(x^K)\|^2 = \cO(\nicefrac{1}{K})$ rate for \algname{EG} without any additional assumptions.

Next, the state-of-the-art last-iterate convergence rates are $\cO(\nicefrac{1}{K^2})$ \citep{kim2021accelerated,yoon2021accelerated}. Moreover, \citet{yoon2021accelerated} derive the optimality of $\cO(\nicefrac{1}{K^2})$ rate in the class of monotone and Lipschitz VIP. Although this rate is better than what we derive for \algname{EG}, this is obtained for different methods (Accelerated \algname{PP} and Anchored \algname{EG}). Since \algname{EG} is one of the most popular methods for solving \eqref{eq:main_problem}, it is important to resolve open questions about it like last-iterate convergence rates. Moreover, in view of the lower bound from \citet{golowich2020last}, our result for the last-iterate convergence of \algname{EG} is optimal for \algname{EG} up to numerical constants.

Finally, we emphasize that it is possible to obtain a linear last-iterate convergence rate when $F$ is additionally strongly monotone. The corresponding results are well-known both for \algname{EG} \citep{tseng1995linear} and \algname{OG} \citep{gidel2019variational, mokhtari2020unified}. Moreover, one can achieve a linear rate under slightly weaker assumptions like quasi-strong monotonicity \citep{loizou2021stochastic}, its local variant (with local guarantees) \citep{azizian2021last}, positive-definiteness of the Jacobian around the solution (also with local guarantees) \citep{hsieh2019convergence}, and error bound (see \citep{hsieh2020explore} and references therein).


\section{COCOERCIVITY AND STAR-COCOERCIVITY}\label{sec:cocoercivity}

In this section, we introduce the main tools connected with cocoercivity. First of all, it is known that cocoercivity is closely related to non-expansiveness in the following sense.


\begin{lemma}[Proposition 4.2 from \citet{bauschke2011convex}]\label{lem:l_non_exp_and_cocoercive}
    For any operator $F:\R^d \to \R^d$ the following are equivalent: (i) $\Id - \frac{2}{\ell}F$ is non-expansive; (ii) $F$ is $\ell$-cocoercive.
\end{lemma}

We use this lemma to prove non-cocoercivity of $F_{\algname{EG}, \gamma}$.

Next, we also study a relaxation of cocoercivity called star-cocoercivity, which turns out to be sufficient to derive best- or random-iterate $\cO(\nicefrac{1}{K})$ rate for \algname{GD}.

\begin{definition}[Star-cocoercivity]
    Operator $F: \R^d \to \R^d$ is called $\ell$-star-cocoercive around $x^*$ if $F(x^*) = 0$ and for all $x \in \R^d$
    \begin{equation}
        \|F(x)\|^2 \le \ell \langle F(x), x - x^* \rangle. \label{eq:l_star_cocoercivity}
    \end{equation}
\end{definition}

Further discussion of cocoercivity and star-cocoercivity is deferred to Appendix~\ref{appendix:cocoercivity}.

\subsection{Analysis of Gradient Descent Under Cocoercivity}\label{sec:GD_cocoercivity}

The simplest method for solving \eqref{eq:main_problem} is Gradient Descent (\algname{GD}):
\begin{equation}
    x^{k+1} = x^k - \gamma F(x^k) \tag{GD}. \label{eq:GD_update}
\end{equation}
If operator $F$ is star-cocoercive, then one can easily show random-iterate convergence of \algname{GD}.


\begin{theorem}[Random-iterate convergence of \algname{GD}]\label{thm:random_iter_conv_GD}
    Let $F:\R^d \to \R^d$ be $\ell$-star-cocoercive around $x^*$. Then for all $K\ge 0$ we have
    \begin{equation}
        \EE\|F(\widehat x^K)\|^2 \le \frac{\ell\|x^0 - x^*\|^2}{\gamma (K+1)}, \label{eq:random_iter_conv_GD}
    \end{equation}
    where $\widehat{x}^K$ is chosen uniformly at random from the set of iterates $\{x^0,x^1,\ldots,x^K\}$ produced by \algname{GD} with $0 < \gamma \le \nicefrac{1}{\ell}$.
\end{theorem}

The proof of this result requires a few lines of simple derivations. Next, to establish last-iterate convergence we need to assume cocoercivity of $F$. In particular, when $F$ is cocoercive it is possible to show that $\{\| F(x^k)\|^2\}_{k\ge 0}$ monotonically decreases. Using this and previous results one can derive last-iterate convergence (see also \citep{diakonikolas2021potential}).
\begin{theorem}[Last-iterate convergence of \algname{GD}]\label{thm:last_iter_conv_GD}
    Let $F:\R^d \to \R^d$ be $\ell$-cocoercive. Then for all $K\ge 0$ we have
    \begin{equation}
        \|F(x^K)\|^2 \le \frac{\ell\|x^0 - x^*\|^2}{\gamma (K+1)}, \label{eq:last_iter_conv_GD}
    \end{equation}
    where $x^K$ is produced by \algname{GD} with $0 < \gamma \le \nicefrac{1}{\ell}$.
\end{theorem}

Overall, the analysis of \algname{GD} under star-cocoercivity and cocoercivity is straightforward and almost identical to the analysis of \algname{GD} for convex smooth minimization.

\subsection{Proximal Point Method}
Consider the following iterative process called Proximal Point method (\algname{PP}):
\begin{equation}
    x^{k+1} = x^k - \gamma F(x^{k+1}). \tag{PP} \label{eq:PP_update_rule}
\end{equation}
That is, the next point $x^{k+1}$ is defined implicitly for given $x^k$ and $\gamma > 0$. Moreover, for given $\gamma > 0$ and any point $x$ we can define operator $F_{\text{PP}, \gamma}: \R^d \to \R^d$ such that $\forall x\in \R^d$
\begin{equation}
    F_{\text{PP},\gamma}(x) = F(y),\quad \text{where} \quad y = x - \gamma F(y). \label{eq:PP_operator}
\end{equation}
Therefore, \eqref{eq:PP_update_rule} can be rewritten as \algname{GD} for $F_{\text{PP}, \gamma}$:
\begin{equation}
    x^{k+1} = x^k - \gamma F_{\text{PP},\gamma}(x^k). \notag 
\end{equation}
It turns out that $F_{\text{PP},\gamma}$ is $\nicefrac{2}{\gamma}$-cocoercive (Corollary 23.10 from \citet{bauschke2011convex}). For completeness, we provide the proof of this fact in the appendix. 

Then, applying Theorem~\ref{thm:last_iter_conv_GD} to the method
\begin{equation}
    x^{k+1} = x^k - \gamma F_{\text{PP}, \nicefrac{2}\ell}(x^k) \tag{PP-$\gamma$-$\ell$}, \label{eq:PP_update_v2}
\end{equation}
we get the following result (see also \citep{gu2019optimal}).
\begin{theorem}[Last-iterate convergence of \eqref{eq:PP_update_v2}]\label{thm:last_iter_conv_PP}
    Let $F: \R^d \to \R^d$ be monotone, $\ell > 0$ and $0 < \gamma \le \nicefrac{1}{\ell}$. Then for all $K\ge 0$ we have
    \begin{equation}
        \EE\|F(\widehat{x}^{K})\|^2 \le \frac{\ell\|x^0 - x^*\|^2}{\gamma (K+1)}, \label{eq:last_iter_conv_PP}
    \end{equation}
    where $\widehat{x}^{K} = x^K - \nicefrac{2}{\ell} F(\widehat{x}^{K}) = x^K - \nicefrac{2}{\ell} F_{\text{PP},\nicefrac{2}{\ell}}(\widehat{x}^{K})$ and $x^K$ is produced by \eqref{eq:PP_update_v2}.
\end{theorem}

\section{EXTRAGRADIENT METHOD}\label{sec:EG}
Inspired by the result on cocoercivity of $F_{\text{PP},\gamma}$, we study Extragradient method (\algname{EG}) through the lens of cocoercivity. Indeed, \algname{EG} can be seen as a practical approximation of \algname{PP} \citep{mokhtari2019proximal, mokhtari2020unified}. Therefore, we consider operator $F_{\algname{EG}, \gamma} = F\left(\Id - \gamma F\right)$ defining the update of \algname{EG}:
\begin{equation}
    x^{k+1} = x^k - \gamma \underbrace{F\left(x^k - \gamma F(x^k)\right)}_{F_{\algname{EG}, \gamma}(x^k)}. \tag{EG}\label{eq:EG_update}
\end{equation}

\paragraph{Affine case.} We start with the case when $F(x)$ is affine, i.e., it can be written as $F(x) = \mA x + b$ for some $\mA \in \R^{d\times d}, b\in \R^d$. In Appendix~\ref{appendix:cocoercive_EG_affine}, we show that $F_{\algname{EG}, \gamma} = F\left(\Id - \gamma F\right)$ is $\nicefrac{2}{\gamma}$-cocoercive for $0 < \gamma \le \nicefrac{1}{L}$. Therefore, applying Theorem~\ref{thm:last_iter_conv_GD} to the method
\begin{equation}
    x^{k+1} = x^k - \gamma_2 F_{\algname{EG}, \gamma_1}(x^k) \tag{EG-$\gamma_1$-$\gamma_2$}, \label{eq:EG_gamma_1_gamma_2}
\end{equation}
one can derive $\cO(\nicefrac{1}{K})$ last-iterate convergence rate in this case (see Appendix~\ref{appendix:last_iter_EG_affine}).

\paragraph{Random-iterate guarantees for \algname{EG}.} Motivated by the positive results on the cocoercivity of $F_{\algname{EG},\gamma}$ in the affine case, below we make an attempt to generalize this approach to the general case. The first result establishes star-cocoercivity of extragradient operator $F_{\algname{EG}, \gamma}$ for any star-monotone and Lipschitz operator $F$.

\begin{lemma}[Star-cocoercivity of extragradient operator]\label{lem:EG_star_cocoercive}
    Let $F:\R^d \to \R^d$ be star-monotone around $x^*$, i.e., $F(x^*) = 0$ and
    \begin{equation}
        \forall x\in\R^d \quad \langle F(x), x - x^* \rangle \ge 0, \label{eq:star_monotonicity_def}
    \end{equation}
    and $L$-Lipschitz. Then, operator $F_{\algname{EG},\gamma} = F\left(\Id - \gamma F\right)$ with $\gamma \le \nicefrac{1}{L}$ is $\nicefrac{2}{\gamma}$-star-cocoercive around $x^*$.
\end{lemma}

Therefore, applying Theorem~\ref{thm:random_iter_conv_GD} to \eqref{eq:EG_gamma_1_gamma_2} we get the following result.
\begin{theorem}[Random-iterate convergence of \eqref{eq:EG_gamma_1_gamma_2}: non-linear case]\label{thm:EG_random_iter_conv}
    Let $F: \R^d \to \R^d$ be star-monotone around $x^*$ and $L$-Lipschitz, $0 < \gamma_2 \le \nicefrac{\gamma_1}{2}$, $0 < \gamma_1 \le \nicefrac{1}{L}$. Then for all $K\ge 0$ we have
    \begin{equation}
        \EE\|F(\widehat{x}^{K})\|^2 \le \frac{2\|x^0 - x^*\|^2}{\gamma_1 \gamma_2(K+1)}, \label{eq:random_iter_conv_EG}
    \end{equation}
    where $\widehat{x}^{K} = \tx^K - \gamma_1 F(\tx^K)$ and $\tx^K$ is chosen uniformly at random from the set of iterates $\{x^0, x^1, \ldots, x^K\}$ produced by \eqref{eq:EG_gamma_1_gamma_2}.
\end{theorem}
We notice that the result is derived under star-monotonicity of $F$ that can hold even for non-monotone $F$ \citep{loizou2021stochastic}.

\paragraph{Non-cocoercivity of \algname{EG} operator.} Taking into account all positive results observed in the previous sections, it is natural to expect that $F_{\algname{EG},\gamma}$ is cocoercive when $F$ is monotone and $L$-Lipschitz and $\gamma$ is sufficiently small. Surprisingly, this is not true in general: $F_{\algname{EG},\gamma}$ can be non-cocoercive for monotone and Lipschitz (and even cocoercive\footnote{Cocoercivity of $F$ implies its monotonicity and Lipschitzness.}) $F$!


In view of Lemma~\ref{lem:l_non_exp_and_cocoercive}, it is sufficient to show that for any $\ell > 0$ and any $\gamma_1, \gamma_2 > 0$ there exists $\ell$-cocoercive operator $F$ such that operator $\Id - \gamma_2 F_{\algname{EG}, \gamma_1}$ is not non-expansive. In other words, our goal is to show that for all $\ell,\gamma_1, \gamma_2 > 0$ the quantity
\begin{eqnarray}
    \rho_{\algname{EG}}(\ell,\gamma_1,\gamma_2) = &\max & \frac{\|\widehat{x} - \widehat{y}\|^2}{\|x-y\|^2} \label{eq:EG_exp_problem}\\
    &\text{s.t.} & F \text{ is $\ell$-cocoercive},\notag \\
    &&x,y\in\R^d,\; x\neq y,\notag\\
    && \widehat{x} = x - \gamma_2 F(x - \gamma_1 F(x)), \notag\\
    && \widehat{y} = y - \gamma_2 F(y - \gamma_1 F(y)) \notag
\end{eqnarray}
is bigger than $1$, i.e., $\rho_{\algname{EG}}(\ell,\gamma_1,\gamma_2) > 1$. In the above problem, maximization is performed on \emph{the set of all $\ell$-cocoercive operators} and pairs of vectors $x,y \in \R^d$, i.e., one needs to solve infinitely dimensional problem. In such form, it is computationally infeasible.

\begin{figure}[t]
    \centering
    \includegraphics[width=0.4\textwidth]{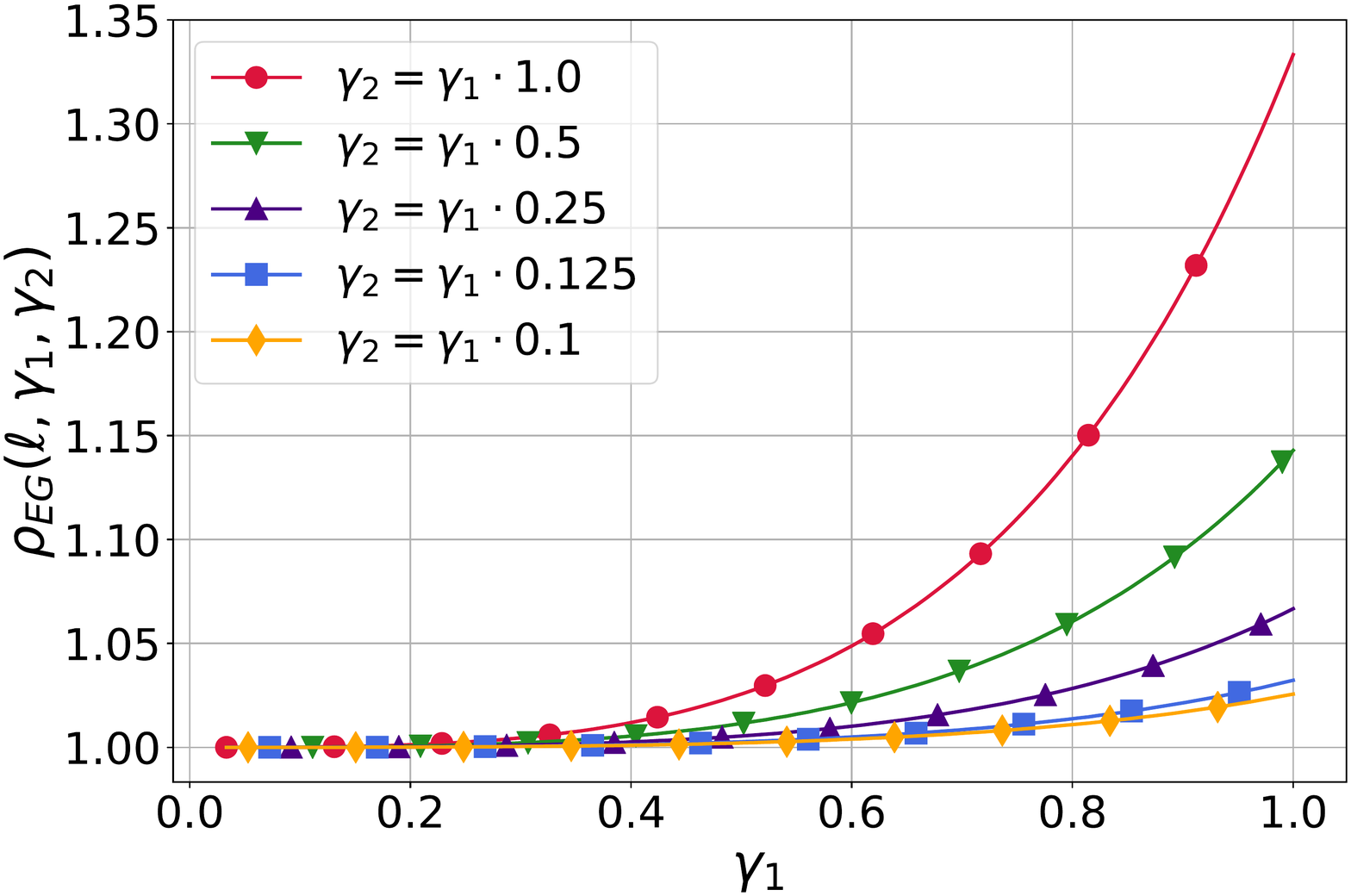}
    \caption{\small
    Numerical estimation of $\rho_{\algname{EG}}(\ell,\gamma_1,\gamma_2)$ defined in \eqref{eq:EG_exp_problem} for $\ell = 1$ and different $\gamma_1,\gamma_2$.}
    \label{fig:EG_expansiveness}
\end{figure}

Fortunately, there exists an equivalent SDP that can be solved efficiently. To construct such a problem, we follow Performance Estimation Problem (PEP) technique from \citet{ryu2020operator} and rewrite \eqref{eq:EG_exp_problem} as
\begin{eqnarray}
    &\max & \frac{\|x - \gamma_2 x_{F_2} - y + \gamma_2 y_{F_2}\|^2}{\|x-y\|^2} \label{eq:EG_exp_problem_1}\\
    &\text{s.t.} & x,y, x_{F_1}, y_{F_1}, x_{F_2}, y_{F_2}\in\R^d,\; x\neq y,\notag\\&&\exists F \text{ is $\ell$-cocoercive}:\notag\\
    && x_{F_2} = F(x - \gamma_1 x_{F_1}),\; x_{F_1} = F(x),\notag\\
    && y_{F_2} = F(y - \gamma_1 y_{F_1}),\; y_{F_1} = F(y).\notag
\end{eqnarray}
Problem \eqref{eq:EG_exp_problem_1} is finite-dimensional and equivalent to \eqref{eq:EG_exp_problem}. Next, for all $\alpha > 0$ the following equivalence holds: $F \text{ is $\ell$-cocoercive} \; \Longleftrightarrow \; \left(\alpha^{-1}\Id\right)\circ F \circ \left(\alpha \Id\right) \text{ is $\ell$-cocoercive}.$ Therefore, in problem \eqref{eq:EG_exp_problem_1} one can apply the change of variables $x := \alpha^{-1}x, y:= \alpha^{-1}y, x_{F_1} := \alpha^{-1}x_{F_1}, y_{F_1} := \alpha^{-1} y_{F_1}, x_{F_2} := \alpha^{-1} x_{F_2}, y_{F_2} := \alpha^{-1} y_{F_2}, F:= \left(\alpha^{-1}\Id\right)\circ F \circ \left(\alpha \Id\right),$ where  $\alpha = \|x-y\|$ and get another equivalent problem: 
\begin{eqnarray}
    &\max & \|x - \gamma_2 x_{F_2} - y + \gamma_2 y_{F_2}\|^2 \label{eq:EG_exp_problem_2}\\
    &\text{s.t.} & x,y, x_{F_1}, y_{F_1}, x_{F_2}, y_{F_2}\in\R^d,\; x\neq y,\notag\\
    &&\|x-y\|^2 = 1 \text{ and }\exists F \text{ is $\ell$-cocoercive}:\notag\\
    && x_{F_2} = F(x - \gamma_1 x_{F_1}),\; x_{F_1} = F(x),\notag\\
    && y_{F_2} = F(y - \gamma_1 y_{F_1}),\; y_{F_1} = F(y).\notag
\end{eqnarray}

However, the constraints are defined implicitly via the existence of $\ell$-cocoercive operator $F$ that interpolates the introduced points. Therefore, the problem is still hard to solve. It turns out (Proposition~2 from \citet{ryu2020operator}) that the constraints about the existence of $\ell$-cocoercive operator are equivalent to the \emph{finite set of inequalities}. These inequalities are called \emph{interpolation conditions} and, essentially, it is inequality \eqref{eq:l_cocoercivity} written for all pairs of points that $F$ has to interpolate. That is, \eqref{eq:EG_exp_problem_2} is equivalent to the following problem:
\begin{eqnarray}
    &\max & \|x - \gamma_2 x_{F_2} - y + \gamma_2 y_{F_2}\|^2 \label{eq:EG_exp_problem_3}\\
    &\text{s.t.} & x,y,x_{F_1}, y_{F_1}, x_{F_2}, y_{F_2}\in\R^d,\; \|x-y\|^2 = 1,\notag\\
    && \ell\langle x_{F_1} - x_{F_2}, \gamma_1 x_{F_1} \rangle \ge \|x_{F_1} - x_{F_2}\|^2,\notag\\
    && \ell\langle x_{F_1} - y_{F_1}, x-y \rangle \ge  \|x_{F_1} - y_{F_1}\|^2,\notag\\
    && \ell\langle x_{F_1} - y_{F_2}, x-y + \gamma_1 y_{F_1}\rangle \ge  \|x_{F_1} - y_{F_2}\|^2,\notag\\
    && \ell\langle x_{F_2} - y_{F_1}, x - \gamma_1 x_{F_1} - y\rangle \ge  \|x_{F_2} - y_{F_1}\|^2,\notag\\
    && \ell\langle x_{F_2} - y_{F_2}, x - \gamma_1 x_{F_1} - y + \gamma_1 y_{F_1}\rangle\notag\\
    &&\hspace{4cm}\ge  \|x_{F_2} - y_{F_2}\|^2, \notag\\
    && \ell\langle y_{F_1} - y_{F_2}, \gamma_1 y_{F_1}\rangle \ge  \|y_{F_1} - y_{F_2}\|^2. \notag
\end{eqnarray}
Although the above representation is much better for numerical solving than \eqref{eq:EG_exp_problem}, we do not stop here and consider a Grammian representation of $\mU = (x,y,x_{F_1}, y_{F_1}, x_{F_2}, y_{F_2})^\top$: $\mG = \mU^\top \mU$. 
One can easily show that for all $d \ge 6$ we have $\mG \in \mathbb{S}_{+}^6$ iff there exist $x,y,x_{F_1}, y_{F_1}, x_{F_2}, y_{F_2} \in \R^d$ such that $\mG$ is Gram matrix for these vectors. Since the objective and constrainsts of \eqref{eq:EG_exp_problem_3} are linear in the entries of matrix $\mG$, problem \eqref{eq:EG_exp_problem_3} is equivalent to the following SDP problem:
\begin{eqnarray}
    &\max & \Tr(\mM_0 \mG) \label{eq:EG_exp_problem_4}\\
    &\text{s.t.} & \mG\in\mathbb{S}_{+}^6,\notag\\
    && \Tr(\mM_i \mG) \ge 0,\; i = 1,2,\ldots,6,\notag\\
    && \Tr(\mM_7 \mG) = 1, \notag
\end{eqnarray}
where $\mM_0,\ldots, \mM_7$ are some symmetric matrices (see the details in Appendix~\ref{sec:details_PEP_EG_non_cocoercive}). 
For any given $\ell,\gamma_1,\gamma_2 > 0$ this problem can be easily solved numerically using PESTO \citep{taylor2017performance}. Therefore, to compute the expansiveness parameter $\rho_{\algname{EG}}(\ell,\gamma_1,\gamma_2)$ of \algname{EG} we solved \eqref{eq:EG_exp_problem_4} for $\ell = 1$ and different values of $\gamma_1,\gamma_2$. The results are reported in Figure~\ref{fig:EG_expansiveness}.

Although these numerical results show that $F_{\algname{EG},\gamma_1}$ can be non-$\nicefrac{2}{\gamma_2}$-cocoercive for different values of $\gamma_1, \gamma_2$, and $\ell = 1$, it is not a rigorous proof that for any $\ell$ and $\gamma_1, \gamma_2 \in (0,\nicefrac{1}{\ell}]$ there exists $\ell$-cocoercive operator $F$ such that $F_{\algname{EG},\gamma_1}$ is not $\nicefrac{2}{\gamma_2}$-cocoercive. Nevertheless, one can utilize numerical results to construct a rigorous proof but it requires to change the problem \eqref{eq:EG_exp_problem_4}.

The main difficulty is that the solution of \eqref{eq:EG_exp_problem_4} is at least of rank $5$ in our experiments. It means, that the dimension of the space where the counter-example $F$ is defined is also at least $5$ complicates the visualization of the solution\footnote{We also tried to solve this problem symbolically, but the problem turned out to be computationally infeasible for standard symbolic solvers. Therefore, we focused on the visualization of the solutions in the hope of finding useful dependencies between the solution and parameters $\ell, \gamma_1, \gamma_2$.}. To overcome this issue, we consider another problem with so-called Log-det heuristic \citep{fazel2003log}:
\begin{eqnarray}
    &\min & \log\det\left(\mG + \delta\mI\right) \label{eq:EG_exp_problem_5}\\
    &\text{s.t.} & \mG\in\mathbb{S}_{+}^6,\notag\\
    && \Tr(\mM_0 \mG) \ge 1.0005, \notag\\
    && \Tr(\mM_i \mG) \ge 0,\; i = 1,2,\ldots,6,\notag\\
    && \Tr(\mM_7 \mG) = 1, \notag
\end{eqnarray}
where $\delta > 0$ is some small positive regularization parameter. For $\gamma_1, \gamma_2, \ell$ in some intervals, the solution of the new problem also provides an example of $x, y$ and operator $F$ that proves non-$\nicefrac{2}{\gamma_2}$-cocoercivity of $F_{\algname{EG},\gamma_1}$: we ensure this via the constraint $\Tr(\mM_0 \mG) \ge a = 1.0005 > 1$. In theory, any $a > 1$ can be used but due to the inevitability of the numerical errors in practice we used $a = 1.0005$. However, due to the change of the objective the solution may have lower rank since $\log\det\left(\mG + \delta\mI\right)$ can be seen as an differentiable approximation of the rank of $\mG$.

\begin{figure*}[t]
    \centering
    \includegraphics[width=0.4\textwidth]{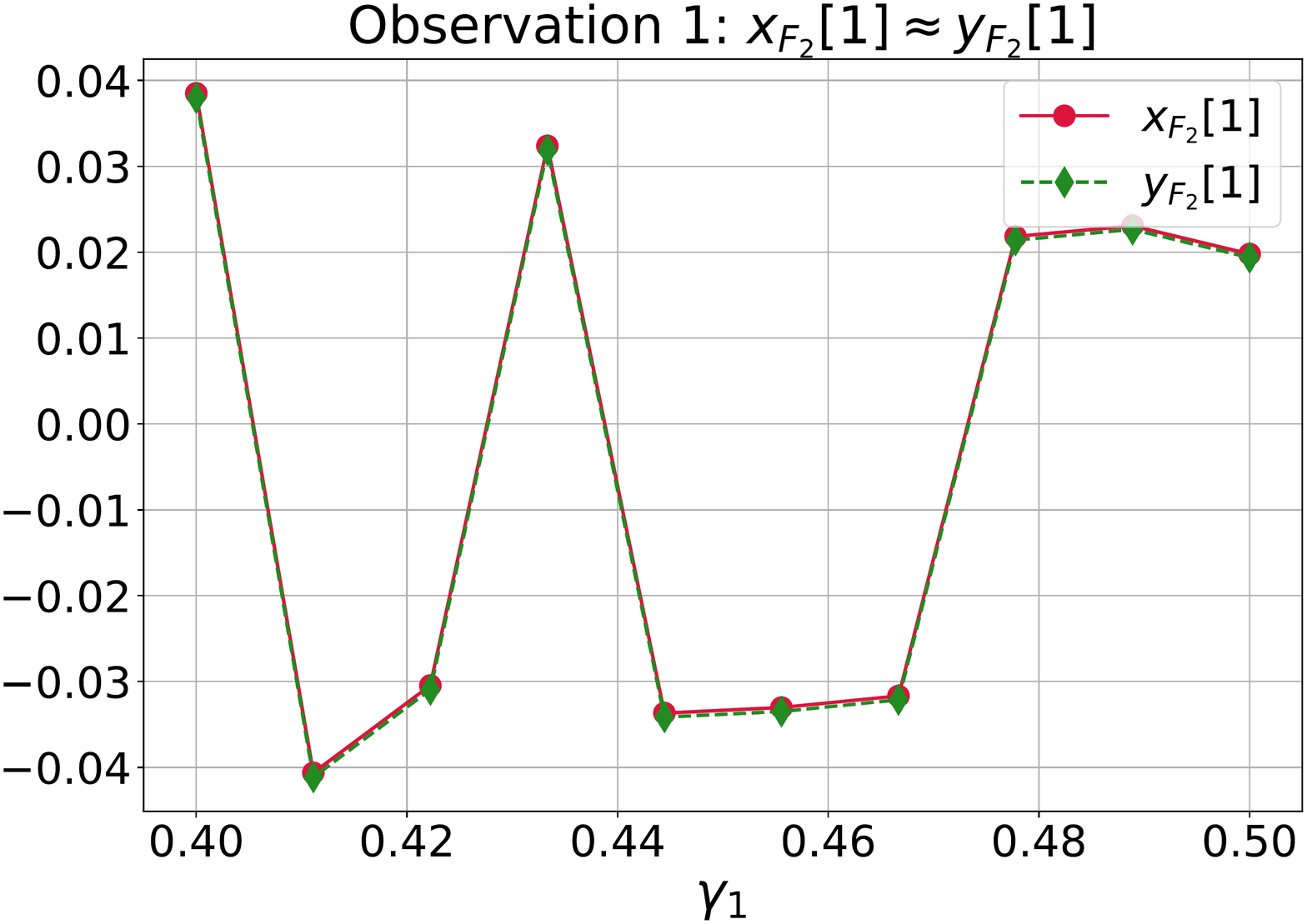}\qquad 
    \includegraphics[width=0.4\textwidth]{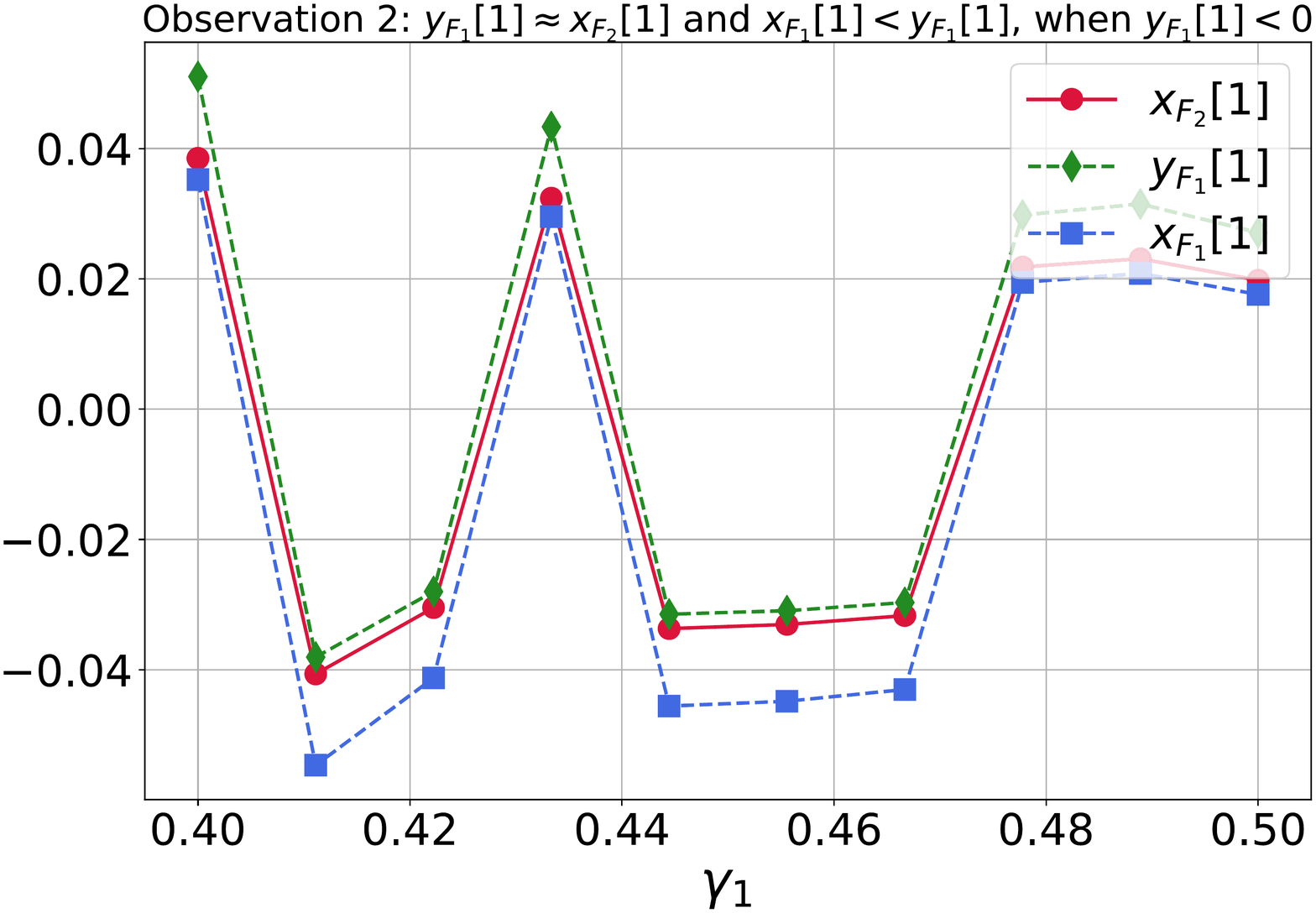}
    \includegraphics[width=0.4\textwidth]{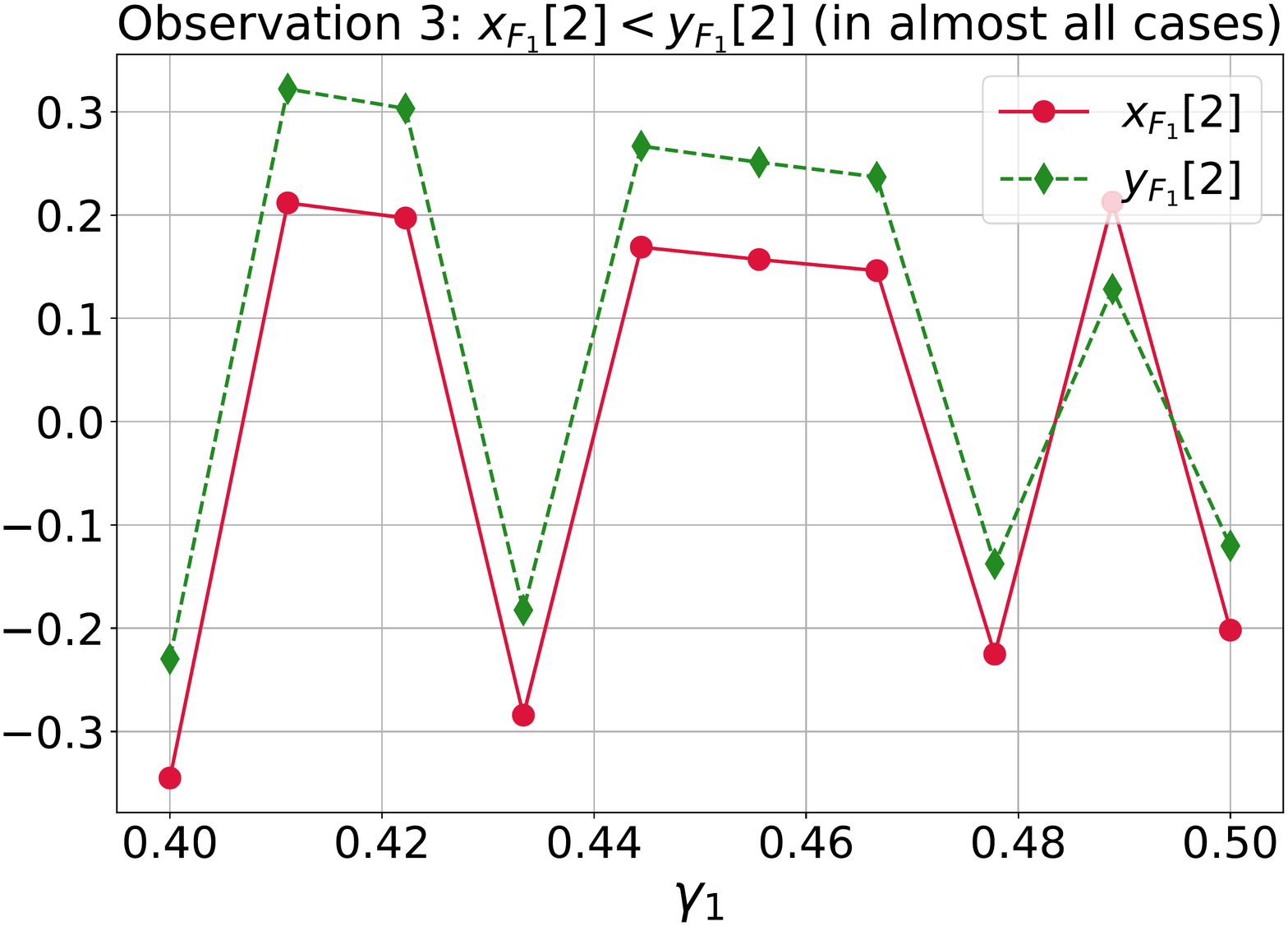}\qquad 
    \includegraphics[width=0.4\textwidth]{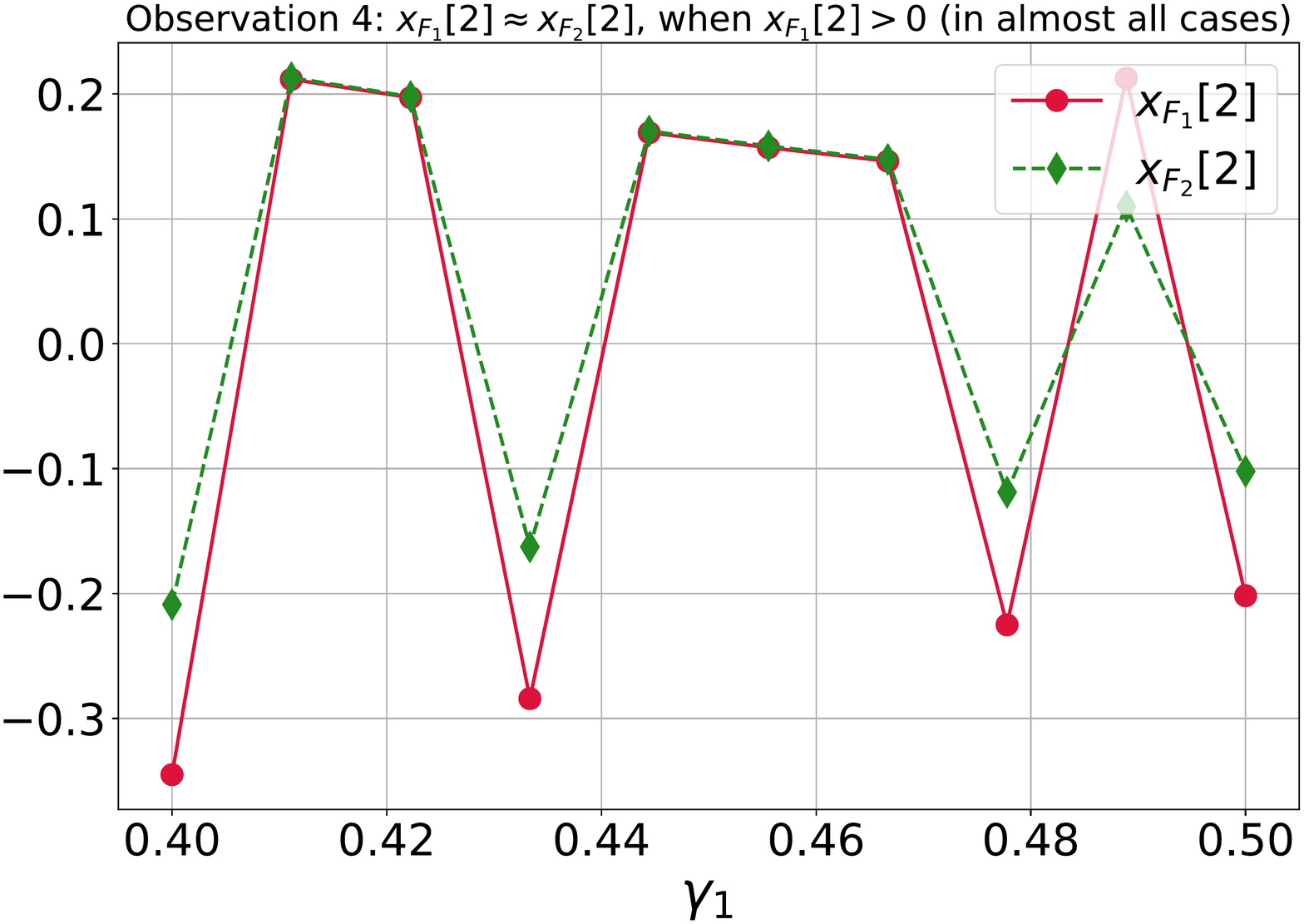}
    \caption{\small
    Observations that we made after plotting the components of $x, y, x_{F_1}, y_{F_1}, x_{F_2}, y_{F_2}$ for $\ell = 1$ and different values of $\gamma_1$ (we used $\gamma_2 = \gamma_1$).}
    \label{fig:EG_observations}
\end{figure*}

Solving problem \eqref{eq:EG_exp_problem_5} for $\ell = 1$, and $\gamma_1 = \gamma_2$, we obtained the solutions of rank $2$, i.e., we obtained $x, y, x_{F_1}, y_{F_1}, x_{F_2}, y_{F_2}$ in $\R^2$. We observed that $x = -y$ for all tested values of $\gamma_1$. However, numerical solutions were not consistent enough to guess the right dependencies. To overcome this issue, we rotated $x, y, x_{F_1}, y_{F_1}, x_{F_2}, y_{F_2}$ in such a way that $x = (-\nicefrac{1}{2}, 0)^\top$, $y = (\nicefrac{1}{2}, 0)^\top$, and plotted the components of $x_{F_1}, y_{F_1}, x_{F_2}, y_{F_2}$ for different $\gamma_1$. Although the resulting dependencies were not perfect, the obtained plots helped us to sequentially construct the needed example:
\begin{gather}
    x = -y = \begin{pmatrix} -\frac{1}{2}\\ 0 \end{pmatrix}\!,\; x_{F_1} = \begin{pmatrix} -\frac{1}{2\gamma_1}\\ \frac{1}{2\gamma_1} \end{pmatrix}\!,\; y_{F_1} = \begin{pmatrix} -\frac{1-\gamma_1\ell}{2\gamma_1}\\ \frac{1+\gamma_1\ell}{2\gamma_1} \end{pmatrix}\!,\notag \\ x_{F_2} = \begin{pmatrix} -\frac{1-\gamma_1\ell}{2\gamma_1}\\ \frac{1}{2\gamma_1} \end{pmatrix}\!,\; y_{F_2} = \begin{pmatrix} -\frac{1-\gamma_1\ell}{2\gamma_1}\\ \frac{1-\gamma_1^2\ell^2}{2\gamma_1} \end{pmatrix}.\label{eq:EG_bad_example}
\end{gather}

That is, via plotting the components of $x, y, x_{F_1}, y_{F_1}, x_{F_2}, y_{F_2}$ we observed $4$ interesting dependencies, see Figure~\ref{fig:EG_observations}. Mimicking these dependencies, we assumed that
\begin{gather*}
    x_{F_2}[1] = y_{F_2}[1],\\
    y_{F_1}[1] = x_{F_2}[1]\quad \text{and}\quad x_{F_1}[1] < y_{F_1}[1] < 0,\\
    0 < x_{F_1}[2] < y_{F_1}[2],\quad x_{F_1}[2] = x_{F_2}[2],
\end{gather*}
plugged these relations in the interpolation conditions from \eqref{eq:EG_exp_problem_3}, and obtained the following inequalities:
\begin{gather*}
    y_{F_1}[1] \leq (1 - \gamma_1)x_{F_1}[1],\quad y_{F_1}[2] \leq \frac{y_{F_2}[2]}{1-\gamma_1},\\
    y_{F_1}[2] \leq (1+\gamma_1)x_{F_2}[2],\quad x_{F_2}[2] \leq \frac{y_{F_2}[2]}{1-\gamma_1^2}
\end{gather*}
To fulfill these constraints, we simply assumed that they hold as equalities and got:
\begin{gather*}
    x_{F_2}[2] = x_{F_1}[2] = \frac{y_{F_2}[2]}{1-\gamma_1^2},\quad y_{F_1}[2] = \frac{y_{F_2}[2]}{1-\gamma_1},\\\
    y_{F_1}[1] = x_{F_2}[1] = y_{F_2}[1] = (1 - \gamma_1)x_{F_1}[1].
\end{gather*}
Using these dependencies in the remaining interpolation conditions, we derived $x_{F_1}[1] + \gamma_1 (x_{F_1}[1])^2 + \frac{\gamma_1 (y_{F_2}[2])^2}{(1-\gamma_1^2)^2} \leq 0.$ After that, we assumed that $y_{F_2}[2] = - x_{F_1}[1](1-\gamma_1^2).$ Together with previous inequality it gives $x_{F_1}[1] + 2\gamma_1 (x_{F_1}[1])^2 \leq 0.$ Next, we chose $x_{F_1} = -\nicefrac{1}{2\gamma_1}$ and put it in all previously derived dependencies. Finally, we generalized the example to the case of non-unit $\ell$ using ``physical-dimension'' arguments and got \eqref{eq:EG_bad_example}.

These derivations lead to the following result that we rigorously prove in Appendix~\ref{sec:details_PEP_EG_non_cocoercive}.

\begin{theorem}\label{thm:EG_bad_example}
    For all $\ell > 0$ and $\gamma_1 \in (0,\nicefrac{1}{\ell}]$ there exists $\ell$-cocoercive operator $F$ such that $F(x) = x_{F_1}, F(y) = y_{F_1}, F(x-\gamma_1 x_{F_1}) = x_{F_2}, F(y - \gamma_1 y_{F_1}) = y_{F_2}$ for $x, y, x_{F_1}, y_{F_1}, x_{F_2}, y_{F_2}$ defined in \eqref{eq:EG_bad_example} and 
    \begin{equation}
        \|x - \gamma_2 F(x - \gamma_1 F(x)) - y + \gamma_2 F(y - \gamma_1 F(y))\| > 1 \label{eq:EG_is_non_cocoercive}
    \end{equation}
    for all $\gamma_2 > 0$, i.e., $F_{\algname{EG},\gamma_1} = F(\Id - \gamma_1F)$ is non-cocoercive.
\end{theorem}

First of all, this result emphasizes the difference between \algname{PP} and \algname{EG}, Moreover, it also means that one cannot apply the technique from Section~\ref{sec:GD_cocoercivity} to prove last-iterate $\cO(\nicefrac{1}{K})$ convergence for \algname{EG} applied to \eqref{eq:main_problem} with monotone and $L$-Lipschitz operator $F$. However, it does not imply that one cannot prove this fact in general.

\paragraph{Last-iterate guarantees for \algname{EG}.} Inspired by the proof of non-cocoercivity of the operator $F_{\algname{EG}, \gamma}$ obtained via PEP, we apply PEP technique to find the rate of convergence in terms of $\|F(x^K)\|^2$ for $F$ being monotone and Lipschitz-continuous. That is, we consider the problem
\begin{eqnarray}
    &\max & \|F(x^K)\|^2 \label{eq:EG_norm_PEP}\\
    &\text{s.t.} & F \text{ is mon.\ and $L$-Lip.},\;x^0\in\R^d,\notag\\
    && \|x^0 - x^*\|^2 \leq 1,\notag\\
    && x^{k+1} = x^k - \gamma_2 F\left(x^k - \gamma_1 F(x^k)\right),\notag\\
    &&k = 0,1,\ldots,K-1 \notag
\end{eqnarray}
and following similar steps to what we do for showing non-cocoercivity of \algname{EG} operator, we construct a special SDP using the definitions of monotonicity \eqref{eq:monotonicity_def} and \eqref{eq:L_lip_def} as interpolation conditions. However, the resulting SDP gives just \textit{an upper bound} for the value of \eqref{eq:EG_norm_PEP} since the resulting SDP might produce such solutions that cannot be interpolated by any monotone and $L$-Lipschitz operator $F$ (see Proposition 3 from \citet{ryu2020operator}). Nevertheless, we solved the resulting SDP using PESTO \citep{taylor2017performance} for $L = 1$, $\gamma_1 = \gamma_2 = \nicefrac{1}{2L}$, and various values of $K$. We observed that the PEP answer behaves as $\cO\left(\nicefrac{1}{K}\right)$ (see Figure~\ref{fig:EG_norm}). Moreover, using standard duality theory for SDP \citep{de2006aspects} one can show that the solution of the dual problem to the SDP obtained from \eqref{eq:EG_norm_PEP} gives the proof of convergence: it is needed just to sum up the constraints with weights corresponding to the solution of the dual problem \citep{de2017worst}. The only thing that remains to do is to guess analytical form of the dual solution.

\begin{figure}[t]
    \centering
    \includegraphics[width=0.41\textwidth]{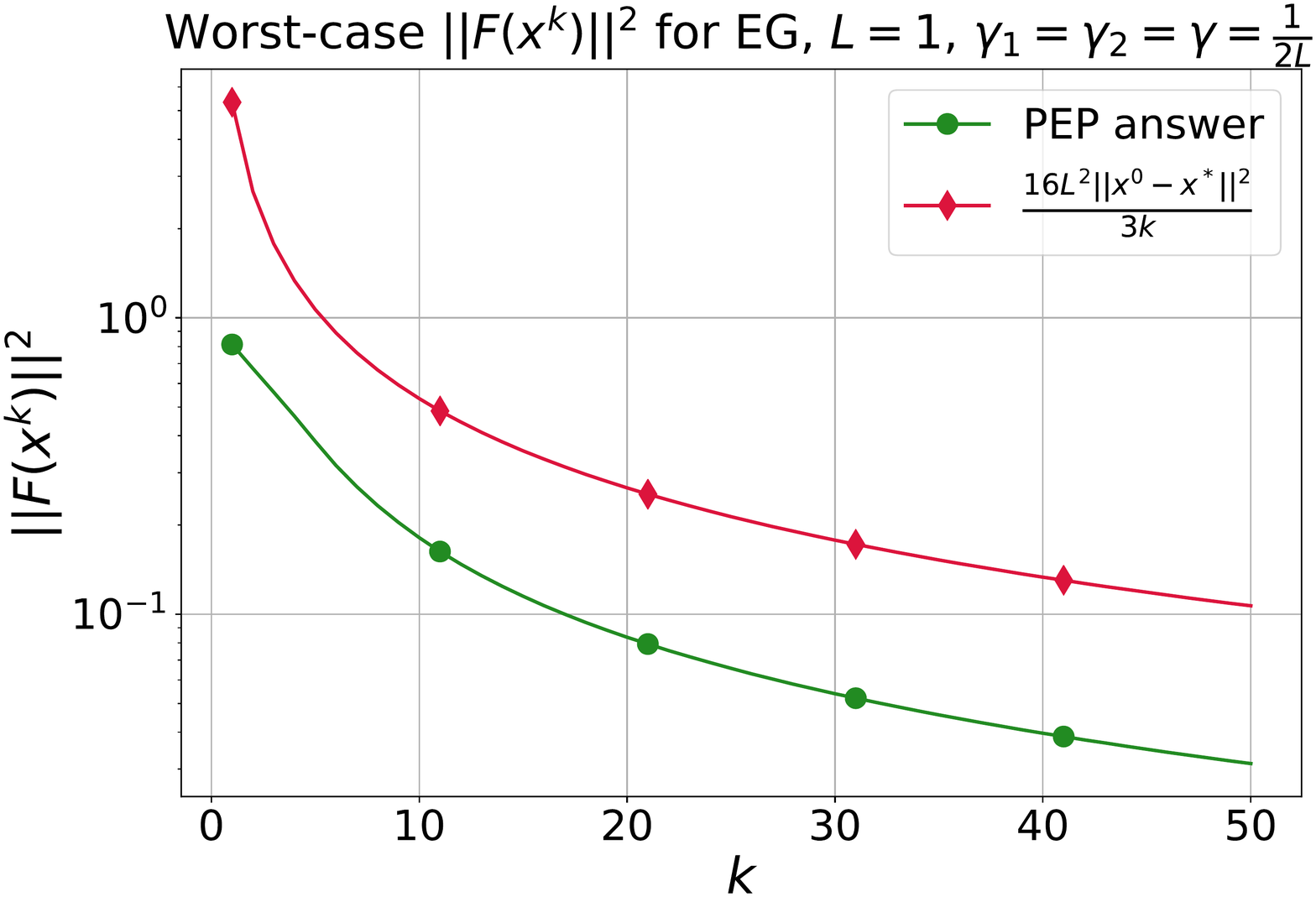}
    \caption{\small
    Comparison of the worst-case rate of \algname{EG} obtained via solving PEP and the guessed upper-bound $\nicefrac{16L^2\|x^0 - x^*\|^2}{k}$. The vertical axis is shown in logarithmic scale and after iteration $k = 20$ the curves are almost parallel, i.e., PEP answer and $\nicefrac{16L^2\|x^0 - x^*\|^2}{k}$ differ almost by a constant factor. In view of Proposition 3 from \citet{ryu2020operator}, PEP may give the answer that is not tight for the class of monotone and Lipschitz operators. However, in this particular case, it turns out to be quite tight.}
    \label{fig:EG_norm}
\end{figure}

However, it is not always an easy task: the dependencies on the parameters of the problem like $L,\gamma_1,\gamma_2$ might be quite tricky. In particular, this might happen due to inaccuracy of the obtained numerical solution and large number of constraints. Therefore, we consider a simpler problem:
\begin{eqnarray}
    \Delta_{\algname{EG}}(L,\gamma_1,\gamma_2) = &\max & \|F(x^1)\|^2 - \|F(x^0)\|^2 \label{eq:EG_norm_diff_PEP}\\
    &\text{s.t.} & F \text{ is mon.\ and $L$-Lip.},\notag\\
    && x^0\in\R^d,\; \|x^0 - x^*\|^2 \leq 1,\notag\\
    && x^{1}\! = \!x^0\! - \!\gamma_2 F\left(x^0 - \gamma_1 F(x^0)\right)\notag
\end{eqnarray}
with $\gamma_1 = \gamma_2 = \gamma$. As for \eqref{eq:EG_norm_PEP}, we construct a corresponding SDP and solve it for different values of $L$ and $\gamma$. In these numerical tests, we observed that $\Delta_{\algname{EG}}(L,\gamma_1,\gamma_2) \approx 0$ for all tested pairs of $L$ and $\gamma$ and the dual variables $\lambda_1, \lambda_2, \lambda_3$ that correspond to 3 particular constraints -- monotonicity \eqref{eq:monotonicity_def} for $(x^k, x^{k+1})$ and $(x^k-\gamma F(x^k), x^{k+1})$ and Lipschitzness for $(x^k-\gamma F(x^k), x^{k+1})$ -- are always very close to $\nicefrac{2}{\gamma}$, $\nicefrac{1}{2\gamma}$, and $\nicefrac{3}{2}$, while other dual variables are negligible. Although $\lambda_2$ and $\lambda_3$ were sometimes slightly smaller, e.g., sometimes we had $\lambda_2 \approx \nicefrac{3}{5\gamma}$ and $\lambda_3 \approx \nicefrac{13}{20}$, we simplified these dependencies and simply summed up the corresponding inequalities with weights $\lambda_1 = \nicefrac{2}{\gamma}$, $\lambda_2 = \nicefrac{1}{2\gamma}$, and $\lambda_3 = \nicefrac{3}{2}$ respectively. After that, it was just needed to rearrange the terms and apply Young's inequality to some inner products. This is how we obtained the following result (see the details in Appendix~\ref{sec:PEP_proof_of_EG_norm_relax}).

\begin{lemma}\label{lem:EG_norm_is_non_incr}
    Let $F: \R^d \to \R^d$ be monotone and $L$-Lipschitz, $0 < \gamma \le \nicefrac{1}{\sqrt{2}L}$. Then for all $k\ge 0$ the iterates produced by \eqref{eq:EG_update} satisfy $\|F(x^{k+1})\| \le \|F(x^k)\|$.
\end{lemma}

This result on its own is novel and plays the central part in deriving last-iterate $\cO(\nicefrac{1}{K})$ rate for \algname{EG} in our analysis. We emphasize that \citet{golowich2020last} do not derive $\|F(x^{k+1})\| \le \|F(x^k)\|$ to show last-iterate $\cO(\nicefrac{1}{K})$ convergence of \algname{EG} and use completely different arguments based on the Lipschitzness of the Jacobian of $F$. Moreover, the assumption on $\gamma$ can be relaxed to $\gamma \leq \nicefrac{1}{L}$, but the proof would be slightly different in this case (though it can be obtained from the same PEP).

Next, one might ask a question: \textit{is it true that $\|F_{\algname{EG},\gamma_1}(x^{k+1})\| \le \|F_{\algname{EG},\gamma_1}(x^{k})\|$ for a reasonable choice of $\gamma_1$ and $\gamma_2$?} Indeed, this is a good question, since the last-iterate $\cO(\nicefrac{1}{K})$ convergence would directly follow from the random-iterate guarantee (Theorem~\ref{thm:EG_random_iter_conv}), if the inequality $\|F_{\algname{EG},\gamma_1}(x^{k+1})\| \le \|F_{\algname{EG},\gamma_1}(x^{k})\|$ held. Perhaps, surprisingly, but this is not true even for $L$-cocoercive $F$: we observed this phenomenon via solving the SDP constructed for
\begin{eqnarray}
    &\max & \|F_{\algname{EG},\gamma_1}(x^1)\|^2 - \|F_{\algname{EG},\gamma_1}(x^0)\|^2 \label{eq:EG_norm_diff_EG_PEP}\\
    &\text{s.t.} & F \text{ is $L$-cocoercive},\;x^0\in\R^d,\notag\\
    && \|x^0 - x^*\|^2 \leq 1,\notag\\
    && x^{1} = x^0 - \gamma_2 F\left(x^0 - \gamma_1 F(x^0)\right)\notag
\end{eqnarray}
with $L = 1$ and $\gamma_1 \in [\nicefrac{1}{4L}, \nicefrac{1}{L}]$, $\gamma_2 \in [\nicefrac{\gamma_1}{4}, \gamma_1]$. In our numerical tests, we observed that the optimal value in the above problem is significantly larger than $0$ for given values of $L,\gamma_1,\gamma_2$. Since it cannot be caused by the inaccuracy of the numerical solution, we conclude that inequality $\|F_{\algname{EG},\gamma_1}(x^{k+1})\| \le \|F_{\algname{EG},\gamma_1}(x^{k})\|$ is violated in some cases.

Moreover, when $\gamma_2 < \gamma_1$ we noticed a similar phenomenon for the norms of $F$. In particular, we solved the SDP constructed for
\begin{eqnarray}
    &\max & \|F(x^1)\|^2 - \|F(x^0)\|^2 \label{eq:EG_norm_diff_different_gamma_PEP}\\
    &\text{s.t.} & F \text{ is $L$-cocoercive},\;x^0\in\R^d,\notag\\
    && \|x^0 - x^*\|^2 \leq 1,\notag\\
    && x^{1} = x^0 - \gamma_2 F\left(x^0 - \gamma_1 F(x^0)\right)\notag
\end{eqnarray}
with $L = 1$, $\gamma_1 \in [\nicefrac{1}{4L}, \nicefrac{1}{L}]$, $\gamma_2 \in [\nicefrac{\gamma_1}{4}, \nicefrac{\gamma_1}{2}]$ and observed that the optimal value in the above problem is significantly larger than $0$ in this case. Therefore, we conclude that using the same stepsizes $\gamma_1 = \gamma_2$ for extrapolation and for the update is crucial for \algname{EG} to have $\|F(x^{k+1})\| \le \|F(x^k)\|$.

To derive the desired last-iterate $\cO(\nicefrac{1}{K})$ convergence of \algname{EG} it remains to combine Lemma~\ref{lem:EG_norm_is_non_incr} with standard arguments for \algname{EG}. 
\begin{theorem}[Last-iterate convergence of \eqref{eq:EG_update}: non-linear case]\label{thm:EG_last_iter_conv_non_linear}
    Let $F: \R^d \to \R^d$ be monotone and $L$-Lipschitz. Then for all $K\ge 0$ and $0 < \gamma \le \nicefrac{1}{\sqrt{2}L}$
    \begin{equation}
        \|F(x^K)\|^2 \le \frac{\|x^0 - x^*\|^2}{\gamma^2(1-L^2\gamma^2)(K+1)}, \label{eq:EG_last_iter_conv_non_linear}
    \end{equation}
    where $x^K$ is produced by \eqref{eq:EG_update} with stepsize $\gamma$, and
    \begin{equation}
        \texttt{Gap}_F(x^K) \le \frac{2\|x^0 - x^*\|^2}{\gamma\sqrt{1-L^2\gamma^2}\sqrt{K+1}}. \label{eq:EG_last_iter_conv_non_linear_gap}
    \end{equation}
\end{theorem}

This is the first result establishing last-iterate rates $\|F(x^K)\|^2 = \cO(\nicefrac{1}{K})$ and $\texttt{Gap}_F(x^K) = \cO(\nicefrac{1}{\sqrt{K}})$ for \algname{EG} that relies on monotonicity and Lipschitzness of $F$ only. Moreover, it matches the lower bounds for \algname{EG} from \citet{golowich2020last}.

\section{OPTIMISTIC GRADIENT METHOD}\label{sec:OG}
As \algname{EG},  Optimistic Gradient method (\algname{OG}) is also often treated as an approximation of \algname{PP}. Therefore, similar questions to those that we study for \algname{EG} arise for \algname{OG}. \algname{OG} can be written in the following way:
\begin{equation}
    x^{k+1} = x^k - 2\gamma F(x^k) + \gamma F(x^{k-1}). \tag{OG}\label{eq:OG_update}
\end{equation}
For the iterates $z^k = ((x^k)^\top, (x^{k-1})^\top)^\top$ \eqref{eq:OG_update} is
\begin{equation}
    z^{k+1}\! = \!z^k \!-\! \gamma F_{\algname{OG},\gamma}(z^k),\; F_{\algname{OG},\gamma} = \begin{pmatrix} 2F & -F \\ -\frac{1}{\gamma}\Id & \frac{1}{\gamma}\Id \end{pmatrix}. \label{eq:OG_update_matrix} 
\end{equation}

There exists another popular form of \eqref{eq:OG_update} called Extrapolation from the past (\algname{EFTP}): $x^0 = \widetilde{x}^0$ and
\begin{equation}
    \widetilde{x}^{k+1} = x^k - \gamma F(\widetilde{x}^k),\; x^{k+1} = x^k - \gamma F(\widetilde{x}^{k+1}). \tag{EFTP} \label{eq:EFTP_update}
\end{equation}
One can show that \eqref{eq:EFTP_update} and \eqref{eq:OG_update} are equivalent:
\begin{eqnarray}
    \widetilde{x}^{k+1} &=& x^k - \gamma F(\widetilde{x}^k) = x^{k-1} - 2\gamma F(\tx^k) \notag\\
    &=& \widetilde{x}^k - 2\gamma F(\widetilde{x}^k) + \gamma F(\widetilde{x}^{k-1}). \label{eq:EFTP_update_to_OG_update}
\end{eqnarray}
However, update rule \eqref{eq:EFTP_update} hints the following matrix representation of the method: for $z^k = ((x^k)^\top, (\widetilde{x}^k)^\top)^\top$
\eqref{eq:EFTP_update} is equivalent to
\begin{eqnarray}
    z^{k+1} &=& z^k - \gamma F_{\algname{EFTP}, \gamma}(z^k),\notag\\
    F_{\algname{EFTP}, \gamma} &=& \begin{pmatrix}F & 0\\ 0 & \Id \end{pmatrix}\begin{pmatrix}\Id & -\gamma F\\ -\frac{1}{\gamma}\Id & \frac{1}{\gamma}\Id + F \end{pmatrix}. \label{eq:EFTP_update_matrix}
\end{eqnarray}

It turns out that for any $\gamma > 0$ operators $F_{\algname{OG},\gamma}$, $F_{\algname{EFTP}, \gamma}$ can be non-star-cocoercive even for $F$ being linear, monotone, and Lipschitz. 

\begin{theorem}\label{thm:F_OG_is_not_star_cocoercive}
    Let the linear operator $F(x) = \mA x$ be monotone and $L$-Lipschitz. Assume that $\Sp(\nabla F(x)) = \Sp(\mA)$ contains at least one eigenvalue $\hat{\lambda}$ such that $\Re(\hat{\lambda}) = 0$ and $\Im(\hat{\lambda}) \neq 0$.  Then, for any $\ell > 0$ and $\gamma > 0$ operators $F_{\algname{OG},\gamma}$, $F_{\algname{EFTP}, \gamma}$ are not $\ell$-star-cocoercive.
\end{theorem}

Therefore, for \textit{the particular representations} \eqref{eq:OG_update_matrix} and \eqref{eq:EFTP_update_matrix} of \eqref{eq:OG_update} one cannot apply the results from Section~\ref{sec:GD_cocoercivity} to derive even random-iterate convergence guarantees.

However, this negative result does not imply that it is impossible to show random-iterate or best-iterate $\cO(\nicefrac{1}{K})$ convergence rate for \algname{OG} or \algname{EFTP}. In fact, such convergence guarantees can be derived using similar steps as in the proof of the corresponding result for \algname{EG}, see Lemma 11 from \citet{golowich2020tight}. Although this result is derived for monotone and Lipschitz operator $F$, the proof uses only star-monotonicity of $F$. For completeness, we provide the complete statement of this result and the full proof in Appendix~\ref{sec:rand_iter_EFTP}.

Moreover, \citet{golowich2020tight} derive $\cO(\nicefrac{1}{K})$ last-iterate convergence rate for \algname{OG}/\algname{EFTP} when $F$ is linear or has Lipschitz Jacobian. Establishing $\cO(\nicefrac{1}{K})$ last-iterate convergence rate for \algname{OG} or \algname{EFTP} for monotone and $L$-Lipschitz operator $F$ without additional assumptions is still an open problem.

\section{CONCLUSION}
In this paper, we close an important gap in the convergence theory for \algname{EG} by showing $\|F(x^K)\|^2 = \cO(\nicefrac{1}{K})$. Our proof is computer-assisted and is based on the PEP technique \citep{taylor2017performance,ryu2020operator}. Moreover, the ideas of reducing the proof to solving SDP problems helped \citet{kim2021accelerated, yoon2021accelerated} to derive last-iterate $\cO(\nicefrac{1}{K^2})$ rates for Accelerated \algname{PP} and Anchored \algname{EG}. We believe that this approach of deriving new proofs is very prominent.

Next, the established connections between \algname{EG}, \algname{OG}, \algname{HGM} and cocoercivity emphasize the differences between these methods and \algname{PP}. This is especially important for \algname{EG} and \algname{OG} that are often treated as similar methods for solving \eqref{eq:main_problem} and as approximations \algname{PP}. Moreover, establishing the result like $\|F(x^K)\|^2 = \cO(\nicefrac{1}{K})$ for \algname{OG} without additional assumptions on $F$ (e.g., without assuming Lipschitzness of Jacobian) remains an open problem.

\subsubsection*{Acknowledgements}
This work was partially supported by a grant for research centers in the field of artificial intelligence, provided by the Analytical Center for the Government of the Russian Federation in accordance with the subsidy agreement (agreement identifier 000000D730321P5Q0002) and the agreement with the Moscow Institute of Physics and Technology dated November 1, 2021 No. 70-2021-00138. Part of this work was done while Nicolas Loizou was a postdoctoral research fellow at Mila,  Université de Montréal, supported by the  IVADO Postdoctoral Funding Program. Gauthier Gidel is supported by an IVADO grant. Part of this work was done while Eduard Gorbunov was an intern at Mila, Universit\'e de Montr\'eal under the supervision of Gauthier Gidel. We thank Adrien Taylor for fruitful discussions, suggestions to consider log-det-heuristic and plot the output of PEP. We also thank Laurent Condat, Konstantin Mishchenko, Adil Salim, and Vladimir Semenov for pointing out important references and useful suggestions regarding the improvement of the text. Finally, we thank anonymous reviewers for their feedback and valuable suggestions regarding improvements to the paper structure.

\bibliography{refs}


\clearpage
\appendix

\thispagestyle{empty}

\onecolumn \makesupplementtitle

{\scriptsize \tableofcontents}

\section{BASIC FACTS AND AN AUXILIARY LEMMA}

In our proofs, we often use the following simple inequalities: for all $a,b \in \R^d$ and $\alpha > 0$
\begin{equation}
    \langle a, b\rangle \le \frac{\alpha}{2}\|a\|^2 + \frac{1}{2\alpha}\|b\|^2, \label{eq:young}
\end{equation}
\begin{equation}
    \|a+b\|^2 \le (1+\alpha)\|a\|^2 + (1+\alpha^{-1})\|b\|^2. \label{eq:a+b}
\end{equation}

Moreover, the following lemma plays a key role in the proof of random-iterate convergence of \eqref{eq:EFTP_update} for star-monotone and Lipschitz continuous \eqref{eq:main_problem}.
\begin{lemma}[Lemma 5 from \citet{gidel2019variational}]\label{lem:EFTP_aux_lemma}
    Let operator $F:\R^d \to \R^d$ be $L$-Lipschitz. Then, for any $x \in \R^d$ the iterates of \eqref{eq:EFTP_update} satisfy
    \begin{equation}
        2\gamma \langle F(\tx^{k+1}), \tx^{k+1} - x \rangle \le \|x^k - x\|^2 - \|x^{k+1} - x\|^2 - \|\tx^{k+1} - x^{k}\|^2 + \gamma^2 L^2 \|\tx^k - \tx^{k+1}\|^2. \label{eq:EFTP_aux_lemma}
    \end{equation}
\end{lemma}

\newpage

\section{ON THE CONVERGENCE RATES UNDER LIPSCHITZNESS OF JACOBIAN}\label{sec:details_on_lip_jac}
As we mention in the main part of the paper, \citet{golowich2020last,golowich2020tight} obtain $\|F(x^K)\|^2 = \cO(\nicefrac{1}{K})$ for \algname{EG} and \algname{OG} when $F$ is monotone and $L$-Lipschitz under additional assumption that $\nabla F$ is $\Lambda$-Lipschitz. Therefore, the result is not applicable to the general case of monotone and $L$-Lipschitz $F$, which can have discontinuous $\nabla F$. Moreover, even in the case of $\Lambda$-Lipschitz Jacobian, the rates from \citet{golowich2020last,golowich2020tight} depend on $\Lambda$ that can be much larger than $L$. Indeed, neglecting numerical factors only, \citet{golowich2020last,golowich2020tight} obtain
\begin{equation}
    \|F(x^K)\|^2 = \cO\left(\frac{L^2\|x^0 - x^*\|^2}{K} + \frac{\Lambda^2\|x^0 - x^*\|^4}{K}\right). \label{eq:rates_under_Jac_lip}
\end{equation}
Consider the logistic loss with a tiny $\ell_2$-regularization:
\begin{equation}
    f(x) = \ln\left(1 + e^{ax}\right) + \frac{\delta}{2} \|x\|^2, \quad a,x\in\R,\quad |a| \gg \delta. \notag
\end{equation}
This function is smooth and strongly convex, therefore, its gradient $F(x) = \nabla f(x)$ is (strongly) monotone and Lipschitz-continuous. Moreover,
\begin{eqnarray*}
    F(x) &=& \frac{a e^{ax}}{1 + e^{ax}} + \delta x,\\
    \nabla F(x) &=& \frac{a^2 e^{ax}}{1+ e^{ax}} - \frac{a^2 e^{2ax}}{(1+e^{ax})^2} + \delta = \frac{a^2 e^{ax}}{(1+e^{ax})^2} + \delta = \frac{a^2}{(e^{-\nicefrac{ax}{2}}+e^{\nicefrac{ax}{2}})^2} + \delta,\\
    \nabla^2 F(x) &=& \frac{a^3 e^{ax}}{(1+e^{ax})^2} - \frac{2a^3 e^{2ax}}{(1+e^{ax})^3} = \frac{a^3 e^{ax}(1 - e^{ax})}{(1+e^{ax})^3} = \frac{a^3}{(e^{-\nicefrac{ax}{2}}+e^{\nicefrac{ax}{2}})^2}\cdot \frac{1-e^{ax}}{1+e^{ax}},
\end{eqnarray*}
and since $\alpha + \alpha^{-1} \ge 2$ for all $\alpha > 0$ we also have
\begin{eqnarray*}
    |\nabla F(x)| &=& \frac{a^2}{(e^{-\nicefrac{ax}{2}}+e^{\nicefrac{ax}{2}})^2} + \delta \le \frac{a^2}{4} + \delta,\\
    |\nabla^2 F(x)| &=& \frac{a^3}{(e^{-\nicefrac{ax}{2}}+e^{\nicefrac{ax}{2}})^2}\cdot \left|\frac{1-e^{ax}}{1+e^{ax}}\right| \le \frac{|a|^3}{4}.
\end{eqnarray*}
Since these upper bounds are not too loose, we have that $L \sim a^2$ and $\Lambda \sim |a|^3$. If additionally $\|x^0 - x^*\| \sim a$, then the second term in the rate from \eqref{eq:rates_under_Jac_lip} is $\sim a^6$ larger than the first one. For example, if $a = 10$, then $\Lambda^2\|x^0 - x^*\|^4$ is larger than $L^2\|x^0 - x^*\|^2$ by $\sim 6$ orders of magnitude. In contrast, our result for last-iterate convergence of \algname{EG} (Theorem~\ref{thm:EG_last_iter_conv_non_linear})
\begin{equation}
    \|F(x^K)\|^2 = \cO\left(\frac{L^2\|x^0 - x^*\|^2}{K}\right) \notag
\end{equation}
is obtained without assuming Lipschitzness of the Jacobian, and, thus, does not suffer from the issues mentioned above.

\newpage

\section{MISSING PROOFS AND DETAILS FROM SECTION \ref{sec:cocoercivity}}\label{appendix:cocoercivity}

\subsection{Proof of Lemma \ref{lem:l_non_exp_and_cocoercive}}
\begin{lemma}[Lemma \ref{lem:l_non_exp_and_cocoercive}; Proposition 4.2 from \citet{bauschke2011convex}]\label{lem:l_non_exp_and_cocoercive_appendix}
    For any operator $F:\R^d \to \R^d$ the following are equivalent
    \begin{enumerate}
        \item[(i)] $\Id - \frac{2}{\ell}F$ is non-expansive.
        \item[(ii)] $F$ is $\ell$-cocoercive.
    \end{enumerate}
\end{lemma}
\begin{proof}
    In fact, Proposition 4.2 from \citet{bauschke2011convex} establishes equivalence of the following statements:
    \begin{enumerate}
        \item[(i)] $\Id - 2F$ is non-expansive.
        \item[(ii)] $F$ is $1$-cocoercive.
    \end{enumerate}
    Therefore, it remains to check how scaling of the operator affect the result. Consider the operator $F_1 = \frac{1}{\ell} F$. This operator has the same solution of \eqref{eq:main_problem} as $F$ and it is $1$-cocoercive:
    \begin{eqnarray*}
        \|F_1(x) - F_1(x')\|^2 &=& \frac{1}{\ell^2}\|F(x) - F(x')\|^2\\
        &\overset{\eqref{eq:l_cocoercivity}}{\le}& \frac{1}{\ell}\langle x - x', F(x) - F(x') \rangle\\
        &=& \langle x - x', F_1(x) - F_1(x') \rangle.
    \end{eqnarray*}
    Moreover, via similar derivation one can show stronger result:
    \begin{equation*}
        F \text{ is } \ell\text{-cocoercive}\quad \Longleftrightarrow\quad  \frac{1}{\ell} F \text{ is } 1\text{-cocoercive}. 
    \end{equation*}
    Applying Proposition 4.2 from \citet{bauschke2011convex}, we obtain
    \begin{equation*}
        F \text{ is } \ell\text{-cocoercive}\quad \Longleftrightarrow\quad  \frac{1}{\ell} F \text{ is } 1\text{-cocoercive} \quad \Longleftrightarrow\quad \Id - \frac{2}{\ell}F \text{ is non-expansive}. 
    \end{equation*}
\end{proof}

\subsection{Proof of Theorem~\ref{thm:random_iter_conv_GD}}
\begin{lemma}[Descent lemma for \algname{GD}] \label{lem:descent_lem_GD}
    Let $F:\R^d \to \R^d$ be $\ell$-star-cocoercive around $x^*$. Then for all $k\ge 0$ iterates produced by \algname{GD} with $\gamma > 0$ satisfy
    \begin{equation}
        \gamma\left(\frac{2}{\ell} - \gamma\right)\|F(x^k)\|^2 \le \|x^k - x^*\|^2 - \|x^{k+1} - x^*\|^2. \label{eq:descent_lem_GD}
    \end{equation}
\end{lemma}
\begin{proof}
    Using the update rule of \eqref{eq:GD_update} we derive
    \begin{eqnarray*}
        \|x^{k+1} - x^*\|^2 &=& \|x^k - \gamma F(x^k) - x^*\|^2\\
        &=& \|x^k - x^*\|^2 - 2\gamma\langle x^k - x^*, F(x^k)\rangle + \gamma^2 \|F(x^k)\|^2\\
        &\overset{\eqref{eq:l_star_cocoercivity}}{\le}& \|x^k - x^*\|^2 - \gamma\left(\frac{2}{\ell} - \gamma\right)\|F(x^k)\|^2.
    \end{eqnarray*}
    Rearranging the terms we get \eqref{eq:descent_lem_GD}.
\end{proof}

Averaging this inequality, one can easily show random-iterate convergence of \algname{GD}.

\begin{theorem}[Theorem~\ref{thm:random_iter_conv_GD}; Random-iterate convergence of \algname{GD}]\label{thm:random_iter_conv_GD_appendix}
    Let $F:\R^d \to \R^d$ be $\ell$-star-cocoercive around $x^*$. Then for all $K\ge 0$ we have
    \begin{equation}
        \EE\|F(\widehat x^K)\|^2 \le \frac{\ell\|x^0 - x^*\|^2}{\gamma (K+1)}, \label{eq:random_iter_conv_GD_appendix}
    \end{equation}
    where $\widehat{x}^K$ is chosen uniformly at random from the set of iterates $\{x^0,x^1,\ldots,x^K\}$ produced by \algname{GD} with $0 < \gamma \le \nicefrac{1}{\ell}$.
\end{theorem}
\begin{proof}
    Summing up inequalities \eqref{eq:descent_lem_GD} for $k = 0,1,\ldots, K$ and dividing both sides of the result by $K+1$ we get
    \begin{eqnarray*}
        \frac{\gamma}{K+1}\left(\frac{2}{\ell} - \gamma\right)\sum\limits_{k=0}^{K}\|F(x^k)\|^2 &\le& \frac{1}{K+1}\sum\limits_{k=0}^{K}\left(\|x^k - x^*\|^2 - \|x^{k+1} - x^*\|^2\right)\\
        &=& \frac{\|x^0 - x^*\|^2 - \|x^{K+1} - x^*\|^2}{K+1}\\
        &\le& \frac{\|x^0 - x^*\|^2}{K+1}.
    \end{eqnarray*}
    Next, we use $\gamma \le \nicefrac{1}{\ell}$ to lower bound $\nicefrac{2}{\ell} - \gamma$ by $\nicefrac{1}{\ell}$ and obtain
    \begin{equation}
        \frac{1}{K+1}\sum\limits_{k=0}^{K}\|F(x^k)\|^2 \le \frac{\ell\|x^0 - x^*\|^2}{\gamma (K+1)}. \label{eq:average_squared_norm_cocoercive}
    \end{equation}
    Finally, since $\widehat{x}^K$ is chosen uniformly at random from the set $\{x^0,x^1,\ldots,x^K\}$ we derive
    \begin{equation*}
        \EE\|F(\widehat x^K)\|^2 = \frac{1}{K+1}\sum\limits_{k=0}^{K}\|F(x^k)\|^2 \le \frac{\ell\|x^0 - x^*\|^2}{\gamma (K+1)}.
    \end{equation*}
\end{proof}

\subsection{Proof of Theorem~\ref{thm:last_iter_conv_GD}}
As we mention in the main part of the paper, Theorem~\ref{thm:last_iter_conv_GD} is a well-known result \citep{brezis1978produits, diakonikolas2021potential}. Moreover, it can derived from the analysis of Krasnoselski-Mann method \citep{krasnosel1955two, mann1953mean}:
\begin{equation*}
    x^{k+1} = \alpha x^k + (1-\alpha) T(x^k),\quad \alpha \in (0,1).
\end{equation*}
Classical results on the convergence of the above method imply that $\|x^{k+1} - T(x^{k+1})\| \leq \|x^k - T(x^k)\|$ for any non-expansive operator $T$ \citep{groetsch1972note, hicks1977mann, borwein1992krasnoselski}. In view of Lemma~\ref{lem:l_non_exp_and_cocoercive}, operator $T = \Id - \tfrac{2}{\ell}F$ is non-expansive for any $\ell$-cocoercive $F$. Moreover, Krasnoselski-Mann method with such operator $T$ is equivalent to \eqref{eq:GD_update} with $\gamma = \nicefrac{2\alpha}{\ell}$ and $x^k - T(x^k) = \tfrac{2}{\ell}F(x^k)$. Therefore, $\|x^{k+1} - T(x^{k+1})\| \leq \|x^k - T(x^k)\|$ implies that $\|F(x^{k+1})\| \le \|F(x^k)\|$.

We give an alternative proof of this fact below.

\begin{lemma}\label{lem:relax_lemma}
    Let $F:\R^d \to \R^d$ be $\ell$-cocoercive. Then for all $k\ge 0$ iterates produced by \algname{GD} with $0 < \gamma \le \nicefrac{2}{\ell}$ satisfy $\|F(x^{k+1})\| \le \|F(x^k)\|$.
\end{lemma}
\begin{proof}
    From cocoercivity we have
    \begin{eqnarray*}
        \|F(x^{k+1}) - F(x^k)\|^2 &\overset{\eqref{eq:l_cocoercivity}}{\le}& \ell\langle F(x^{k+1}) - F(x^k), x^{k+1} - x^k \rangle\\
        &=& - \gamma\ell\langle F(x^{k+1}), F(x^k)\rangle + \gamma\ell \|F(x^k)\|^2.
    \end{eqnarray*}
    Expanding the square in the left-hand side of the inequality and rearranging the terms we get
    \begin{eqnarray}
        \|F(x^{k+1})\|^2 &\le& (2 - \gamma\ell) \langle F(x^{k+1}), F(x^k) \rangle - (1 - \gamma\ell)\|F(x^k)\|^2\notag\\
        &=& \|F(x^k)\|^2 - (2-\gamma\ell)\langle F(x^k) - F(x^{k+1}), F(x^k)\rangle\notag\\
        &\overset{\eqref{eq:GD_update}}{=}& \|F(x^k)\|^2 - \frac{2-\gamma\ell}{\gamma}\langle F(x^k) - F(x^{k+1}), x^k - x^{k+1}\rangle. \label{eq:relax_technical}
    \end{eqnarray}
    Since $0 < \gamma \le \nicefrac{2}{\ell}$ and $F$ is cocoercive we have
    \begin{equation*}
        \frac{2-\gamma\ell}{\gamma}\langle F(x^k) - F(x^{k+1}), x^k - x^{k+1}\rangle \ge \frac{2-\gamma\ell}{\ell\gamma}\|F(x^k) - F(x^{k+1})\|^2 \ge 0.
    \end{equation*}
    Plugging this into \eqref{eq:relax_technical} gives $\|F(x^{k+1})\| \le \|F(x^k)\|$.
\end{proof}

\begin{theorem}[Theorem~\ref{thm:last_iter_conv_GD}; Last-iterate convergence of \algname{GD}]\label{thm:last_iter_conv_GD_appendix}
    Let $F:\R^d \to \R^d$ be $\ell$-cocoercive. Then for all $K\ge 0$ we have
    \begin{equation}
        \|F(x^K)\|^2 \le \frac{\ell\|x^0 - x^*\|^2}{\gamma (K+1)}, \label{eq:last_iter_conv_GD_appendix}
    \end{equation}
    where $x^K$ is produced by \algname{GD} with $0 < \gamma \le \nicefrac{1}{\ell}$.
\end{theorem}
\begin{proof}
    Since cocoercivity implies star-cocoercivity we have
    \begin{equation*}
        \frac{1}{K+1}\sum\limits_{k=0}^{K}\|F(x^k)\|^2 \overset{\eqref{eq:average_squared_norm_cocoercive}}{\le} \frac{\ell\|x^0 - x^*\|^2}{\gamma (K+1)}.
    \end{equation*}
    From Lemma~\ref{lem:relax_lemma} we have that $\|F(x^{k+1})\| \le \|F(x^k)\|$. Putting all together we get \eqref{eq:last_iter_conv_GD}.
\end{proof}

\subsection{Proof of Theorem~\ref{thm:last_iter_conv_PP}}
\begin{lemma}[Corollary 23.10 from \citet{bauschke2011convex}; Cocoercivity of Proximal Point operator]\label{lem:PP_cocoercive}
    Let $F: \R^d \to \R^d$ be monotone and $\gamma > 0$. Then operator $F_{\text{PP},\gamma}$ defined in \eqref{eq:PP_operator} is $\nicefrac{2}{\gamma}$-cocoercive.
\end{lemma}
\begin{proof}
In view of Lemma~\ref{lem:l_non_exp_and_cocoercive}, it is enough to prove that $\Id - \gamma F_{\text{PP},\gamma}$ is non-expansive. To show this we consider arbitrary $x,y\in \R^d$ and define $\widehat{x}$ and $\widehat{y}$ as follows:
\begin{equation}
    \widehat{x} = x - \gamma F(\widehat{x}) = x - \gamma F_{\text{PP},\gamma}(x),\quad \widehat{y} = y - \gamma F(\widehat{y}) = y - \gamma F_{\text{PP},\gamma}(y). \notag
\end{equation}
Using this notation, we derive
\begin{eqnarray*}
    \|\widehat{x} - \widehat{y}\|^2 &=& \|x - y\|^2 -2\gamma \langle x-y, F(\widehat{x}) - F(\widehat{y}) \rangle + \gamma^2 \|F(\widehat{x}) - F(\widehat{y})\|^2\\
    &=& \|x - y\|^2 -2\gamma \langle \widehat{x} + \gamma F(\widehat{x}) - \widehat{y} - \gamma F(\widehat{y}), F(\widehat{x}) - F(\widehat{y}) \rangle + \gamma^2 \|F(\widehat{x}) - F(\widehat{y})\|^2\\
    &=& \|x-y\|^2 -2\gamma \langle \widehat{x} - \widehat{y}, F(\widehat{x}) - F(\widehat{y}) \rangle - \gamma^2 \|F(\widehat{x}) - F(\widehat{y})\|^2\\
    &\overset{\eqref{eq:monotonicity_def}}{\le}& \|x-y\|^2 - \gamma^2 \|F(\widehat{x}) - F(\widehat{y})\|^2\\
    &\le& \|x-y\|^2.
\end{eqnarray*}
That is, $\Id - \gamma F_{\text{PP},\gamma}$ is non-expansive, and, as a result, $F_{\text{PP},\gamma}$ is $\nicefrac{2}{\gamma}$-cocoercive.
\end{proof}

\begin{theorem}[Theorem~\ref{thm:last_iter_conv_PP}; Last-iterate convergence of \eqref{eq:PP_update_v2}]\label{thm:last_iter_conv_PP_appendix}
    Let $F: \R^d \to \R^d$ be monotone, $\ell > 0$ and $0 < \gamma \le \nicefrac{1}{\ell}$. Then for all $K\ge 0$ we have
    \begin{equation}
        \EE\|F(\widehat{x}^{K})\|^2 \le \frac{\ell\|x^0 - x^*\|^2}{\gamma (K+1)}, \label{eq:last_iter_conv_PP_appendix}
    \end{equation}
    where $\widehat{x}^{K} = x^K - \nicefrac{2}{\ell} F(\widehat{x}^{K}) = x^K - \nicefrac{2}{\ell} F_{\text{PP},\nicefrac{2}{\ell}}(\widehat{x}^{K})$ and $x^K$ is produced by \eqref{eq:PP_update_v2}.
\end{theorem}
\begin{proof}
    Theorem~\ref{thm:last_iter_conv_GD} implies
    \begin{equation*}
        \|F_{\text{PP},\nicefrac{2}{\ell}}(x^K)\|^2 \le \frac{\ell\|x^0 - x^*\|^2}{\gamma (K+1)}.
    \end{equation*}
    Since by definition of $F_{\text{PP},\nicefrac{2}{\ell}}$ we have $F_{\text{PP},\nicefrac{2}{\ell}}(x^K) = F(\widehat{x}^K)$, \eqref{eq:last_iter_conv_PP} holds.
\end{proof}

\subsection{Further Details on Cocoercivity and Star-Cocoercivity}

In Section~\ref{sec:cocoercivity}, we give the main definitions and results about cocoercivity and star-cocoercivity that we use in the paper. Here we continue this discussion and provide extra details on these properties of the operator.

As for cocoercivity, there is a certain relation between star-cocoercivity and non-expansiveness around $x^*$.

\begin{definition}[Non-expansiveness around $x^*$]
    Let $x^*\in \R^d$ be some point. Then operator $U:\R^d \to \R^d$ is called non-expansive around $x^*$ if for all $x \in \R^d$
        \begin{equation}
            \|U(x) - U(x^*)\| \le \|x - x^*\|. \label{eq:non_exp_star}
        \end{equation}
\end{definition}

\begin{lemma}\label{lem:star_non_exp_and_star_cocoercive}
    For any operator $F:\R^d \to \R^d$ and $x^*$ such that $F(x^*) = 0$ the following are equivalent:
    \begin{enumerate}
        \item[(i)] $\Id - 2F$ is non-expansive around $x^*$.
        \item[(ii)] $F$ is $1$-star-cocoercive around $x^*$.
    \end{enumerate}
\end{lemma}
\begin{proof}
    Non-expansiveness of $\Id - 2F$ around $x^*$ is equivalent to
    \begin{equation*}
        \|x - x^* - 2(F(x) - F(x^*))\|^2 \le \|x - x^*\|^2
    \end{equation*}
    that is equivalent to
    \begin{equation*}
        \|x - x^*\|^2 - 4\langle x - x^*, F(x)\rangle + 4\|F(x)\|^2 \le \|x - x^*\|^2.
    \end{equation*}
    Rearranging the terms, we get that the last inequality coincides with \eqref{eq:l_star_cocoercivity} for $\ell = 1$.
\end{proof}

\begin{lemma}\label{lem:l_star_non_exp_and_star_cocoercive}
    For any operator $F:\R^d \to \R^d$ and $x^*$ such that $F(x^*) = 0$ the following are equivalent:
    \begin{enumerate}
        \item[(i)] $\Id - \frac{2}{\ell}F$ is non-expansive around $x^*$.
        \item[(ii)] $F$ is $\ell$-star-cocoercive around $x^*$.
    \end{enumerate}
\end{lemma}
\begin{proof}
    The proof is identical to the proof of Lemma~\ref{lem:l_non_exp_and_cocoercive} up to the replacement of $x'$ by $x^*$.
\end{proof}

Finally, we provide a connection between cocoercivity and star-cocoercivity. It is clear that the former implies the latter. Here the natural question arises: \textit{when the opposite implication is true?} To answer this question we consider the class of linear operators.

\begin{definition}[Linear operator]
    We say that operator $F:\R^d \to \R^d$ is linear if for any $\alpha,\beta\in \R$ and $x,y \in \R^d$ the operator satisfies $F(\alpha x + \beta y) = \alpha F(x) + \beta F(y)$.
\end{definition}

It turns out that for linear operators cocoercivity and star-cocoercivity are equivalent.

\begin{lemma}\label{lem:cocoer_equiv_star_cocoer_for_lin_ops}
    For any linear operator $F: \R^d \to \R^d$ the following are equivalent:
    \begin{itemize}
        \item[(i)] $F$ is $\ell$-cocoercive.
        \item[(ii)] $F$ is $\ell$-star-cocoercive around $x^*$.
    \end{itemize}
\end{lemma}
\begin{proof}
    Implication (i) $\Longrightarrow$ (ii) holds always. Therefore, we need to prove that (ii) implies (i). Let $F$ be $\ell$-star-cocoercive around $x^*$, i.e., $F(x^*) = 0$ and the following inequality holds for all $x\in\R^d$:
    \begin{equation*}
        \|F(x)\|^2 \le \ell\langle F(x), x - x^* \rangle.
    \end{equation*}
    Next, due to linearity of $F$ we have $F(x) = F(x) - F(x^*) = F(x - x^*)$ for all $x\in \R^d$. Therefore, for all $x \in \R^d$
    \begin{equation*}
        \|F(x-x^*)\|^2 \le \ell\langle F(x-x^*), x - x^* \rangle.
    \end{equation*}
    For any $y \in \R^d$ one can take $x = y+x^*$ in the above inequality and get
    \begin{equation*}
        \|F(y)\|^2 \le \ell\langle F(y), y \rangle.
    \end{equation*}
    Finally, consider arbitrary $x,x' \in \R^d$. Replacing $y$ with $x-x'$ and using linearity of $F$, we derive
    \begin{eqnarray*}
        \|F(x) - F(x')\|^2 &=& \|F(x-x')\|^2\\
        &\le& \ell\langle F(x-x'), x-x' \rangle\\
        &=& \ell\langle F(x) - F(x'), x-x' \rangle,
    \end{eqnarray*}
    i.e., $F$ is $\ell$-cocoercive.
\end{proof}

We rely on this fact when deriving non-star-cocoercivity of two naturally arising operators corresponding to \algname{OG}.

\subsection{Spectral Viewpoint on Cocoercivity}\label{sec:spectral_view}
The following result establishes the connection between cocoercivity and the spectrum of the Jacobian. This result is a corollary of Proposition~1 from \citet{ryu2021scaled}. For completeness, we provide our proof in the appendix.
\begin{lemma}[Spectrum in a disk]
\label{prop:spectum_ell}
Let $F:\R^d \to \R^d$ be a continuously differentiable operator. Then the following statements are equivalent:
\begin{equation}
    F(x) \text{ is $\ell$-cocoercive} \quad \Longleftrightarrow \quad \Re(\nicefrac{1}{\lambda}) \geq \nicefrac{1}{\ell} \,,\quad \forall \lambda \in \Sp (\nabla F(x)) \,,\; \forall x \in\R^d\,.\label{eq:cocoercivity_and_spectrum}
\end{equation}
As Figure~\ref{fig: eg zone convergence} shows, such a constraint corresponds to a disk centered in $\nicefrac{\ell}{2}$ and of radius $\nicefrac{\ell}{2}$.
\end{lemma}
\begin{proof} We start with proving $(\Rightarrow)$ part of \eqref{eq:cocoercivity_and_spectrum}. Let us consider $x,u \in \R^d$, by $\ell$-cocoercivity we have,
\begin{equation}
  \|F(x)-F(x + tu)\|^2 \leq \ell t\langle F(x)-F(x + tu),u\rangle \,,\quad \forall t >0 \,.\notag
\end{equation}
Divinding both side by $t^2$ and letting $t$ goes to $0$ gives 
\begin{equation}
  \|\nabla F(x)u\|^2 \leq \ell \langle\nabla F(x) u,u\rangle  \,.\notag
\end{equation}
Now let us consider $u = a+ib$ where $a,b \in \R^d$ an eigenvector of $\nabla F(x)$, we get
\begin{equation}
  |\lambda|^2 \|u\|^2 = \|\nabla F(x)u\|^2 = \|\nabla F(x)a\|^2 + \|\nabla F(x)b\|^2  \leq \ell ( \langle \nabla F(x)a ,a\rangle + \langle \nabla F(x)b ,b\rangle) \,,\notag
\end{equation}
where the last inequality comes from the co-coercivity applied twice.
Now let us notice that since $u$ is an eigenvector, we have
\begin{equation}
\left\{ 
    \begin{aligned}
    \nabla F(x)a = \Re(\lambda) a - \Im(\lambda) b, \notag\\
    \nabla F(x)b = \Im(\lambda) a + \Re(\lambda) b. \notag
    \end{aligned}
    \right.
\end{equation}
Thus we get 
\begin{equation}
\langle \nabla F(x)a ,a\rangle + \langle \nabla F(x)b ,b\rangle = \Re(\lambda) (\|a\|^2 + \|b\|^2) \notag
\end{equation}
which leads to,
\begin{equation}
\frac{|\lambda|^2}{\Re(\lambda)} \leq \ell \quad \Longleftrightarrow \quad \Re(\nicefrac{1}{\lambda}) \geq \nicefrac{1}{\ell}. \notag
\end{equation}
We notice that $\Re(\nicefrac{1}{\lambda}) \geq \nicefrac{1}{\ell}$ is equivalent to $\lambda \in \cD_{\nicefrac{\ell}{2}}(\nicefrac{\ell}{2}) = \{\lambda \in \CC \mid |\lambda - \nicefrac{\ell}{2}| \le \nicefrac{\ell}{2}\}$.

Next, we establish $(\Leftarrow)$ part of \eqref{eq:cocoercivity_and_spectrum}. Let $\Sp(\nabla F(x)) \subseteq \cD_{\nicefrac{\ell}{2}}(\nicefrac{\ell}{2})$ for all $x\in \R^d$. In view of Lemma~\ref{lem:l_non_exp_and_cocoercive}, it is sufficient to show that $\Id - \nicefrac{2}{\ell} F$ is non-expansive that is equivalent to $\Sp(\mI - \nicefrac{2}{\ell} \nabla F(x)) \subseteq \cD_{1}(0)$ for all $x \in \R^d$. Moreover, we have
\begin{equation*}
    \Sp\left(\mI - \frac{2}{\ell} \nabla F(x)\right) = \left\{1 - \frac{2}{\ell}\lambda \mid \lambda \in \Sp(\nabla F(x))\right\} \subseteq \left\{1 - \frac{2}{\ell}\lambda \mid \lambda \in \cD_{\nicefrac{\ell}{2}}(\nicefrac{\ell}{2})\right\} .
\end{equation*}
Finally, for any $\lambda \in \cD_{\nicefrac{\ell}{2}}(\nicefrac{\ell}{2})$ we have
\begin{eqnarray*}
    \left|1 - \frac{2}{\ell}\lambda \right| &=& \frac{2}{\ell} \left| \frac{\ell}{2} - \lambda \right| \le \frac{2}{\ell} \cdot \frac{\ell}{2} = 1,
\end{eqnarray*}
i.e., $1 - \frac{2}{\ell}\lambda \in \cD_{1}(0)$. This finishes the proof.
\end{proof}

We use this lemma to show cocoercivity of $F_{\algname{EG}, \gamma}$ when $F$ is affine.

\begin{figure}[H]
	\centering
	\newcounter{t}
	\newcounter{t2}
	\newcounter{L}
	\newcounter{mu}
	\setcounter{t}{90}
	\setcounter{L}{10}
	\setcounter{mu}{5}
	\setcounter{t2}{60}
	\begin{tikzpicture}[scale=.8]
	\begin{axis}[
	grid=major,
	axis equal image,
	yticklabel={
		$\pgfmathprintnumber{\tick}i$
	},
	xmin=-1,   xmax={\value{L} + 1},
	ymin={-\value{L}/2 - 1},   ymax={\value{L}/2 + 4},
	]
	\draw [domain=-\value{t}:\value{t}, draw=yellow, thick, fill=yellow!50!white, fill opacity=0.5, samples=65] plot (axis cs: {\value{L}*cos(\x)}, {\value{L}*sin(\x)});
	\draw [draw=yellow, thick](axis cs: {\value{L}*cos(\value{t})}, {\value{L}*sin(\value{t})}) -- (axis cs: {\value{L}*cos(\value{t})}, {\value{L}*sin(-\value{t})});
	\draw [domain=-180:180, draw=red, fill=red!50!white, thick, fill opacity=0.5, samples=65] plot (axis cs: {\value{L}*(1 + cos(\value{t}))/2 + cos(\x) * \value{L}*(1 - cos(\value{t}))/2}, {sin(\x) * \value{L}*(1 - cos(\value{t}))/2});
	\node[anchor=south west] (mu) at (axis cs: {\value{L} * cos(\value{t})}, 0) {$0$};
	\draw [fill=black] (axis cs: {\value{L} * cos(\value{t})}, 0) circle (1pt);
	\node [circle, anchor=south west] (L) at (axis cs: {\value{L}}, 0) {$\ell$};
	\node [circle, anchor=north] at (axis cs: {\value{L}/2+2.35}, 2 ) {$\ell$-cocoercive};
	\node [circle] at (axis cs: {\value{L}/2}, {\value{L}/2+2} ) {\parbox{3cm}{Monotone \\ \& $\ell$-Lipschitz}};
	\draw [fill=black] (axis cs: {\value{L}}, 0) circle (1pt);
	\draw [draw=blue, thick] (axis cs:0,0) -- (axis cs:{\value{L}},0);
	\end{axis}
	\end{tikzpicture}
	\caption{\small 
		Illustration of the constraint on the spectrum of the Jacobian of a $\ell$-cocoercive operator. In yellow, the constraint for the eigenvalues of the Jacobian $\Sp(\nabla F(x))$ of a monotone and $\ell$-Lipschitz operator are shown. Red region $\Re(1/\lambda) \geq 1/\ell$ corresponds to the constraints for the eigenvalues $\lambda$ of the Jacobian of a $\ell$-cocoercive operator (Lemma~\ref{prop:spectum_ell}).
	}\label{fig: eg zone convergence}
\end{figure}
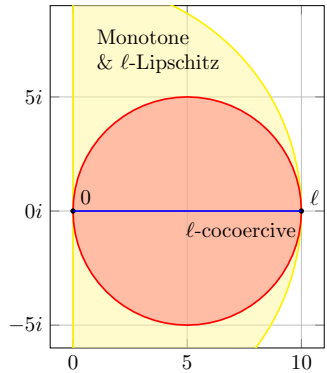

\newpage

\section{MISSING PROOFS AND DETAILS FROM SECTION~\ref{sec:EG}}

\subsection{Cocoercivity of \algname{EG} Operator in the Affine Case}\label{appendix:cocoercive_EG_affine}

\begin{lemma}\label{lem:cocoercivity_of_EG_linear_spec_appendix}
    Let $F:\R^d \to \R^d$ be affine, monotone and $L$-Lipschitz operator. Then, for all $\lambda \in \Sp(\nabla F_{\algname{EG}, \gamma})$ we have $\Re(\nicefrac{1}{\lambda}) \ge \nicefrac{\gamma}{2}$ for all $0< \gamma \le \nicefrac{1}{L}$. In view of Lemma~\ref{prop:spectum_ell}, this implies that $F_{\algname{EG}, \gamma} = F\left(\Id - \gamma F\right)$ is $\nicefrac{2}{\gamma}$-cocoercive for all $0< \gamma \le \nicefrac{1}{L}$.
\end{lemma}
\begin{proof}
    Since $F = \mA x + b$ is monotone and $L$-Lipschitz, we have
    \begin{equation*}
        \Sp(\nabla F) = \Sp(\mA) \subseteq \{\lambda \in \CC\mid \Re(\lambda) \ge 0\; \&\;  |\lambda| \le L\}.
    \end{equation*}
    Next, $F_{\algname{EG},\gamma}(x) = \mA\left(x - \gamma \mA x - \gamma b\right) + b = \mA\left(\mI - \gamma\mA\right)x - \gamma\mA b + b$ and
    \begin{equation*}
        \Sp(\nabla F_{\algname{EG},\gamma}) = \Sp\left(\mA\left(\mI - \gamma\mA\right)\right) = \{\lambda(1-\gamma\lambda)\mid \lambda \in \Sp(\mA)\}.
    \end{equation*}
    Therefore, it is sufficient to prove
    \begin{equation*}
        \{\lambda(1-\gamma\lambda)\mid \Re(\lambda) \ge 0 \; \& \; |\lambda| \le L\} \subseteq \cD_{\nicefrac{1}{\gamma}}\left(\nicefrac{1}{\gamma}\right) := \left\{\lambda \in \CC \mid \left|\lambda - \nicefrac{1}{\gamma}\right| \le \nicefrac{1}{\gamma}\right\},
    \end{equation*}
    since $\Re(\nicefrac{1}{\lambda}) \ge \nicefrac{\gamma}{2}$ is equivalent to $\lambda \in \cD_{\nicefrac{1}{\gamma}}\left(\nicefrac{1}{\gamma}\right)$. In the remaining part of the proof, we will show even stronger result:
    \begin{equation}
        \{\lambda(1-\gamma\lambda)\mid \Re(\lambda), \Im(\lambda) \in [0,L]\} \subseteq \cD_{\nicefrac{1}{\gamma}}\left(\nicefrac{1}{\gamma}\right). \label{eq:cocoercivity_of_EG_linear_spec_technical_1}
    \end{equation}
    Consider arbitrary $\lambda = \lambda_0 + i\lambda_1$ such that $\lambda_0,\lambda_1 \in [0,L]$. Then,
    \begin{eqnarray*}
        \lambda(1-\gamma\lambda) &=& \left(\lambda_0 + i\lambda_1\right)\left(1 - \gamma\lambda_0 - i\gamma\lambda_1\right)\\
        &=& \lambda_0(1-\gamma\lambda_0) + \gamma\lambda_1^2 + i\left(\lambda_1(1-\gamma\lambda_0) - \gamma\lambda_0\lambda_1\right)\\
        &=& \lambda_0(1-\gamma\lambda_0) + \gamma\lambda_1^2 + i\lambda_1\left(1-2\gamma\lambda_0\right),
    \end{eqnarray*}
    implying that
    \begin{eqnarray*}
        |\lambda(1-\gamma\lambda) - \nicefrac{1}{\gamma}|^2 &=& \left(\frac{\gamma\lambda_0(1-\gamma\lambda_0)+ \gamma^2\lambda_1^2 - 1}{\gamma}\right)^2 + \lambda_1^2(1-2\gamma\lambda_0)^2\\
        &\overset{x:= \lambda_0,\; y := \lambda_1^2}{=}& \left(\frac{\gamma x(1-\gamma x)+ \gamma^2 y - 1}{\gamma}\right)^2 + y(1-2\gamma x)^2.
    \end{eqnarray*}
    One can notice that the expression above is a convex function of $y$. Since $0\le y\le L^2$, we have
    \begin{eqnarray*}
        |\lambda(1-\gamma\lambda) - \nicefrac{1}{\gamma}|^2 \le \max\left\{\left(\frac{\gamma x(1-\gamma x) - 1}{\gamma}\right)^2, \left(\frac{\gamma x(1-\gamma x)+ \gamma^2 L^2 - 1}{\gamma}\right)^2 + L^2(1-2\gamma x)^2\right\}.
    \end{eqnarray*}
    Since $x \in [0,L]$ and $\gamma \le \nicefrac{1}{L}$ we have $0 \le x(1-\gamma x) \le x \le L \le \frac{1}{\gamma}$ implying
    \begin{eqnarray*}
        \left(\frac{\gamma x(1-\gamma x) - 1}{\gamma}\right)^2 &\le& \frac{1}{\gamma^2}.  
    \end{eqnarray*}
    Next, we consider the second term in the maximum as a function of $x$:
    \begin{eqnarray*}
        f(x) &=& \left(\frac{\gamma x(1-\gamma x)+ \gamma^2 L^2 - 1}{\gamma}\right)^2 + L^2(1-2\gamma x)^2\\
        &=& \left(-\gamma x^2 + x + \gamma L^2 - \frac{1}{\gamma}\right)^2 + L^2(1 - 4\gamma x + 4\gamma^2 x^2)\\
        &=& \gamma^2 x^4 + x^2 + \gamma^2\left(L^2 - \frac{1}{\gamma^2}\right)^2 - 2\gamma x^3 -2\gamma^2 x^2\left(L^2 - \frac{1}{\gamma^2}\right) + 2\gamma x\left(L^2 - \frac{1}{\gamma^2}\right)\\
        &&\quad + L^2 - 4\gamma L^2 x + 4\gamma^2 L^2 x^2\\
        &=& \gamma^2 x^4 - 2\gamma x^3 + x^2\left(1 + 2\gamma^2\left(L^2 + \frac{1}{\gamma^2}\right)\right) -2\gamma x\left(L^2 + \frac{1}{\gamma^2}\right) + L^2 + \gamma^2\left(L^2 - \frac{1}{\gamma^2}\right)^2.
    \end{eqnarray*}
    Since for all $x \in [0,L]$
    \begin{eqnarray*}
        f''(x) &=& 12\gamma^2 x^2 - 12\gamma x + 2 + 4\gamma^2\left(L^2 + \frac{1}{\gamma^2}\right)\\
        &=& 12\gamma^2 \left(x^2 - \frac{x}{\gamma} + \frac{1}{6\gamma^2} + \frac{L^2}{3} + \frac{1}{3\gamma^2}\right)\\
        &=& 12\gamma^2 \left(\left(x - \frac{1}{2\gamma}\right)^2 + \frac{L^2}{3} + \frac{1}{4\gamma^2}\right) > 0,
    \end{eqnarray*}
    function $f(x)$ is convex. Therefore,
    \begin{eqnarray*}
        f(x) &\le& \max\left\{f(0), f(L)\right\}\\
        &=& \max\left\{\left(\gamma L^2 - \frac{1}{\gamma}\right)^2 + L^2, \left(\frac{\gamma L(1-\gamma L)+ \gamma^2 L^2 - 1}{\gamma}\right)^2 + L^2(1-2\gamma L)^2\right\}\\
        &=& \max\left\{\left(\gamma L^2 - \frac{1}{\gamma}\right)^2 + L^2, \left(L - \frac{1}{\gamma}\right)^2 + L^2(1-2\gamma L)^2\right\}.
    \end{eqnarray*}
    Since $\gamma \le \nicefrac{1}{L}$, we have
    \begin{eqnarray*}
        \left(\gamma L^2 - \frac{1}{\gamma}\right)^2 + L^2 &=& \gamma^2 L^4 - 2L^2 + \frac{1}{\gamma^2} + L^2\\
        &=& \frac{1}{\gamma^2} + L^2\left(\gamma^2 L^2 - 1\right) \le \frac{1}{\gamma^2},
    \end{eqnarray*}
    and
    \begin{eqnarray*}
        \left(L - \frac{1}{\gamma}\right)^2 + L^2(1-2\gamma L)^2 &=& L^2 - \frac{2L}{\gamma} + \frac{1}{\gamma^2} + L^2 - 4\gamma L^3 + 4\gamma^2 L^4\\
        &=& \frac{1}{\gamma^2} + \frac{2L}{\gamma}\left(\gamma L - 1\right) + 4\gamma L^3(\gamma L - 1) \le \frac{1}{\gamma^2}.
    \end{eqnarray*}
    Putting all together, we get $f(x) \le \nicefrac{1}{\gamma^2}$ and, as a result, $|\lambda(1-\gamma\lambda) - \nicefrac{1}{\gamma}|^2 \le \nicefrac{1}{\gamma^2}$. Therefore, \eqref{eq:cocoercivity_of_EG_linear_spec_technical_1} holds. This finishes the proof.
\end{proof}

\subsection{Last-Iterate Convergence of \algname{EG} in the Affine Case}\label{appendix:last_iter_EG_affine}

\begin{theorem}[Last-iterate convergence of \eqref{eq:EG_gamma_1_gamma_2}: affine case]\label{thm:EG_last_iter_affine_appendix}
    Let $F: \R^d \to \R^d$ be affine, monotone and $L$-Lipschitz, $0 < \gamma_2 \le \nicefrac{\gamma_1}{2}$, $0 < \gamma_1 \le \nicefrac{1}{L}$. Then for all $K\ge 0$ we have
    \begin{equation}
        \EE\|F(\widehat{x}^{K})\|^2 \le \frac{2\|x^0 - x^*\|^2}{\gamma_1 \gamma_2(K+1)}, \label{eq:last_iter_conv_EG_appendix}
    \end{equation}
    where $\widehat{x}^{K} = x^K - \gamma_1 F(x^K)$ and $x^K$ is produced by \eqref{eq:EG_gamma_1_gamma_2}.
\end{theorem}
\begin{proof}
    Lemma~\ref{lem:cocoercivity_of_EG_linear_spec_appendix} and Theorem~\ref{thm:last_iter_conv_GD} imply
    \begin{equation*}
        \|F_{\algname{EG},\gamma_1}(x^K)\|^2 \le \frac{2\|x^0 - x^*\|^2}{\gamma_1\gamma_2 (K+1)}.
    \end{equation*}
    Since by definition of $F_{\algname{EG},\gamma_1}$ we have $F_{\algname{EG},\gamma_1}(x^K) = F(x^k - \gamma_1 F(x^K)) = F(\widehat{x}^{K})$, \eqref{eq:last_iter_conv_EG_appendix} holds.
\end{proof}

\subsection{Linear Case: Non-Spectral Analysis of Extragradient Method}\label{appendix:cocoercivity_EG_linear}

In this subsection, we give an alternative proof of cocoercivity of extragradient operator when $F$ is linear, monotone and $L$-Lipschitz.
\begin{lemma}\label{lem:cocoercivity_of_EG_linear_non_spec}
    Let $F:\R^d \to \R^d$ be linear, monotone and $L$-Lipschitz operator. Then, for all $\gamma \le \nicefrac{1}{L}$ extragradient operator $F_{\algname{EG}, \gamma} = F\left(\Id - \gamma F\right)$ is $\nicefrac{2}{\gamma}$-cocoercive.
\end{lemma}
\begin{proof}
    Lemma~\ref{lem:l_non_exp_and_cocoercive} implies that it is sufficient to prove non-expansiveness of $\Id - \gamma F_{\algname{EG},\gamma}$. Consider arbitrary $x,y \in\R^d$ and define
    \begin{equation*}
        \tx = x - \gamma F(x),\quad \ty = y - \gamma F(y),\quad \hx = x - \gamma F(\tx),\quad \hy = y - \gamma F(\ty).
    \end{equation*}
    Our goal is to show that $\|\hx - \hy\| \le \|x - y\|$. Using the monotonicity of $F$ and
    \begin{equation}
        2\langle a,b \rangle = \|a\|^2 + \|b\|^2 - \|b-a\|^2 \label{eq:a_b_inequality}
    \end{equation}
    that holds for all $a,b\in\R^d$, we derive
    \begin{eqnarray}
        \|\hx - \hy\|^2 &=& \|x-y\|^2 - 2\gamma\langle x - y, F(\tx) - F(\ty) \rangle + \gamma^2\|F(\tx) - F(\ty)\|^2 \notag\\
        &=& \|x - y\|^2 -2\gamma \langle \tx + \gamma F(x) -\ty - \gamma F(y), F(\tx) - F(\ty)\rangle + \gamma^2\|F(\tx) - F(\ty)\|^2 \notag\\
        &=& \|x-y\|^2 -2\gamma \langle \tx -\ty, F(\tx) - F(\ty)\rangle \notag\\
        &&\quad -\gamma^2\left(2\langle F(x) - F(y), F(\tx) - F(\ty)\rangle -\|F(\tx) - F(\ty)\|^2\right)\notag\\
        &\overset{\eqref{eq:monotonicity_def},\eqref{eq:a_b_inequality}}{\le}& \|x-y\|^2 + \gamma^2\left(\|F(\tx) - F(\ty) - F(x) + F(y)\|^2 - \|F(x) - F(y)\|^2\right).\label{eq:EG_non_exp_non_spectral_technical_2}
    \end{eqnarray}
    Next, since $F$ is linear and $L$-Lipschitz, we have
    \begin{eqnarray*}
        \|F(\tx) - F(\ty) - F(x) + F(y)\|^2 - \|F(x) - F(y)\|^2 &=& \|F(\tx - x) - F(\ty - y)\|^2 - \|F(x) - F(y)\|^2\\
        &=& \|F(\gamma F(x)) - F(\gamma F(y))\|^2 - \|F(x) - F(y)\|^2\\
        &\overset{\eqref{eq:L_lip_def}}{\le}& \left(L^2\gamma^2-1\right)\|F(x) - F(y)\|^2\\
        &\le& 0,
    \end{eqnarray*}
    where in the final step we apply $\gamma \le \nicefrac{1}{L}$. Putting this inequality in \eqref{eq:EG_non_exp_non_spectral_technical_2} we obtain $\|\hx - \hy\|^2 \le \|x - y\|^2$ that finishes the proof.
\end{proof}

\subsection{Proof of Lemma~\ref{lem:EG_star_cocoercive}}

\begin{lemma}[Lemma~\ref{lem:EG_star_cocoercive}; Star-cocoercivity of extragradient operator]\label{lem:EG_star_cocoercive_appendix}
    Let $F:\R^d \to \R^d$ be star-monotone around $x^*$, i.e., $F(x^*) = 0$ and
    \begin{equation}
        \forall x\in\R^d \quad \langle F(x), x - x^* \rangle \ge 0, \label{eq:star_monotonicity_def_appendix}
    \end{equation}
    and $L$-Lipschitz. Then, extragradient operator $F_{\algname{EG},\gamma} = F\left(\Id - \gamma F\right)$ with $\gamma \le \nicefrac{1}{L}$ is $\nicefrac{2}{\gamma}$-star-cocoercive around $x^*$.
\end{lemma}
\begin{proof}
    Lemma~\ref{lem:l_star_non_exp_and_star_cocoercive} implies that
    \begin{equation*}
        F_{\algname{EG},\gamma} \text{ is } \frac{2}{\gamma}\text{-star-cocoercive around $x^*$} \quad \Longleftrightarrow\quad \Id - \gamma F_{\algname{EG},\gamma}  \text{ is non-expansive around $x^*$}.
    \end{equation*}
    Consider arbitrary $x \in\R^d$ and define
    \begin{equation*}
        \tx = x - \gamma F(x),\quad \hx = x - \gamma F(\tx).
    \end{equation*}
    Since $F_{\algname{EG},\gamma}(x^*) = 0$, our goal is to show that $\|\hx - x^*\| \le \|x - x^*\|$. Using star-monotonicity of $F$ and \eqref{eq:a_b_inequality}, we derive
    \begin{eqnarray}
        \|\hx - x^*\|^2 &=& \|x-x^*\|^2 - 2\gamma\langle x - x^*, F(\tx)\rangle + \gamma^2\|F(\tx)\|^2 \notag\\
        &=& \|x - x^*\|^2 -2\gamma \langle \tx + \gamma F(x) - x^*, F(\tx)\rangle + \gamma^2\|F(\tx)\|^2 \notag\\
        &=& \|x-x^*\|^2 -2\gamma \langle \tx -x^*, F(\tx) )\rangle -\gamma^2\left(2\langle F(x), F(\tx)\rangle -\|F(\tx)\|^2\right)\notag\\
        &\overset{\eqref{eq:star_monotonicity_def},\eqref{eq:a_b_inequality}}{\le}& \|x-x^*\|^2 + \gamma^2\left(\|F(\tx) - F(x)\|^2 - \|F(x)\|^2\right).\label{eq:EG_star_non_exp_technical_1}
    \end{eqnarray}
    Next, since $F$ is $L$-Lipschitz, we have
    \begin{eqnarray*}
        \|F(\tx) - F(x)\|^2 - \|F(x)\|^2 &\overset{\eqref{eq:L_lip_def}}{\le}& L^2\|\tx - x\|^2 - \|F(x)\|^2\\
        &=&\left(L^2\gamma^2-1\right)\|F(x)\|^2\\
        &\le& 0,
    \end{eqnarray*}
    where in the final step we apply $\gamma \le \nicefrac{1}{L}$. Putting this inequality in \eqref{eq:EG_star_non_exp_technical_1} we obtain $\|\hx - x^*\|^2 \le \|x - x^*\|^2$ that finishes the proof.
\end{proof}

\subsection{Details on Performance Estimation Problem for Showing Non-Cocoercivity of Extragradient Operator}\label{sec:details_PEP_EG_non_cocoercive}

First of all, we provide the formulas for the matrices $\mM_0, \ldots, \mM_7$ defining problem \eqref{eq:EG_exp_problem_4}:
\begin{equation*}
    \mM_0 = \begin{pmatrix}1 & -1 & 0 & 0 & -\gamma_2 & \gamma_2\\ 
    -1 & 1 & 0 & 0 & \gamma_2 & -\gamma_2\\
    0 & 0 & 0 & 0 & 0 & 0\\
    0 & 0 & 0 & 0 & 0 & 0\\
    -\gamma_2 & \gamma_2 & 0 & 0 & \gamma_2^2 & -\gamma_2^2\\
    \gamma_2 & -\gamma_2 & 0 & 0 & -\gamma_2^2 & \gamma_2^2\end{pmatrix},\quad \mM_1 = \begin{pmatrix} 0 & 0 & 0 & 0 & 0 & 0\\ 
    0 & 0 & 0 & 0 & 0 & 0\\
    0 & 0 & \ell\gamma_1 -1 & 0 & 1 - \frac{\ell\gamma_1}{2} & 0\\
    0 & 0 & 0 & 0 & 0 & 0\\
    0 & 0 & 1 - \frac{\ell\gamma_1}{2} & 0 & -1 & 0\\
    0 & 0 & 0 & 0 & 0 & 0\end{pmatrix},
\end{equation*}
\begin{equation*}
    \mM_2 = \begin{pmatrix} 0 & 0 & \frac{\ell}{2} & -\frac{\ell}{2} & 0 & 0\\ 
    0 & 0 & -\frac{\ell}{2} & \frac{\ell}{2} & 0 & 0\\
    \frac{\ell}{2} & -\frac{\ell}{2} & -1 & 1 & 0 & 0\\
    -\frac{\ell}{2} & \frac{\ell}{2} & 1 & -1 & 0 & 0\\
    0 & 0 & 0 & 0 & 0 & 0\\
    0 & 0 & 0 & 0 & 0 & 0\end{pmatrix},\quad \mM_3 = \begin{pmatrix} 0 & 0 & \frac{\ell}{2} & 0 & 0 & -\frac{\ell}{2}\\ 
    0 & 0 & -\frac{\ell}{2} & 0 & 0 & \frac{\ell}{2}\\
    \frac{\ell}{2} & -\frac{\ell}{2} & -1 & \frac{\ell\gamma_1}{2} & 0 & 1\\
    0 & 0 & \frac{\ell\gamma_1}{2} & 0 & 0 & -\frac{\ell\gamma_1}{2}\\
    0 & 0 & 0 & 0 & 0 & 0\\
    -\frac{\ell}{2} & \frac{\ell}{2} & 1 & -\frac{\ell\gamma_1}{2} & 0 & -1\end{pmatrix},
\end{equation*}
\begin{equation*}
    \mM_4 = \begin{pmatrix} 0 & 0 & 0 & -\frac{\ell}{2} & \frac{\ell}{2} & 0\\ 
    0 & 0 & 0 & \frac{\ell}{2} & -\frac{\ell}{2} & 0\\
    0 & 0 & 0 & \frac{\ell\gamma_1}{2} & -\frac{\ell\gamma_1}{2} & 0\\
    -\frac{\ell}{2} & \frac{\ell}{2} & \frac{\ell\gamma_1}{2} & -1 & 1 & 0\\
    \frac{\ell}{2} & -\frac{\ell}{2} & -\frac{\ell\gamma_1}{2} & 1 & -1 & 0\\
    0 & 0 & 0 & 0 & 0 & 0\end{pmatrix},\quad \mM_5 = \begin{pmatrix} 0 & 0 & 0 & 0 & \frac{\ell}{2} & -\frac{\ell}{2}\\ 
    0 & 0 & 0 & 0 & -\frac{\ell}{2} & \frac{\ell}{2}\\
    0 & 0 & 0 & 0 & -\frac{\ell\gamma_1}{2} & \frac{\ell\gamma_1}{2}\\
    0 & 0 & 0 & 0 & \frac{\ell\gamma_1}{2} & -\frac{\ell\gamma_1}{2}\\
    \frac{\ell}{2} & -\frac{\ell}{2} & -\frac{\ell\gamma_1}{2} & \frac{\ell\gamma_1}{2} & -1 & 1\\
    -\frac{\ell}{2} & \frac{\ell}{2} & \frac{\ell\gamma_1}{2} & -\frac{\ell\gamma_1}{2} & 1 & -1\end{pmatrix},
\end{equation*}
\begin{equation*}
    \mM_6 = \begin{pmatrix} 0 & 0 & 0 & 0 & 0 & 0\\ 
    0 & 0 & 0 & 0 & 0 & 0\\
    0 & 0 & 0 & 0 & 0 & 0\\
    0 & 0 & 0 & \ell\gamma_1 - 1 & 0 & 1 - \frac{\ell\gamma_1}{2}\\
    0 & 0 & 0 & 0 & 0 & 0\\
    0 & 0 & 0 & 1 - \frac{\ell\gamma_1}{2} & 0 & -1\end{pmatrix},\quad \mM_7 = \begin{pmatrix} 1 & -1 & 0 & 0 & 0 & 0\\ 
    -1 & 1 & 0 & 0 & 0 & 0\\
    0 & 0 & 0 & 0 & 0 & 0\\
    0 & 0 & 0 & 0 & 0 & 0\\
    0 & 0 & 0 & 0 & 0 & 0\\
    0 & 0 & 0 & 0 & 0 & 0\end{pmatrix}.
\end{equation*}

Next, we provide a rigorous proof that the example \eqref{eq:EG_bad_example} is valid, i.e., we prove Theorem~\ref{thm:EG_bad_example}.
\begin{theorem}[Theorem~\ref{thm:EG_bad_example}]\label{thm:EG_bad_example_appendix}
    For all $\ell > 0$ and $\gamma_1 \in (0,\nicefrac{1}{\ell}]$ there exists $\ell$-cocoercive operator $F$ such that $F(x) = x_{F_1}, F(y) = y_{F_1}, F(x-\gamma_1 x_{F_1}) = x_{F_2}, F(y - \gamma_1 y_{F_1}) = y_{F_2}$ for $x, y, x_{F_1}, y_{F_1}, x_{F_2}, y_{F_2}$ defined in \eqref{eq:EG_bad_example} and 
    \begin{equation}
        \|x - \gamma_2 F(x - \gamma_1 F(x)) - y + \gamma_2 F(y - \gamma_1 F(y))\| > 1 = \|x-y\| \label{eq:EG_is_non_cocoercive_appendix}
    \end{equation}
    for all $\gamma_2 > 0$, i.e., $F_{\algname{EG},\gamma_1} = F(\Id - \gamma_1F)$ is non-cocoercive.
\end{theorem}
\begin{proof}
    Proposition~2 from \citet{ryu2020operator} implies that it is sufficient to show that
    \begin{eqnarray}
    \ell\langle x_{F_1} - x_{F_2}, \gamma_1 x_{F_1} \rangle &\ge& \|x_{F_1} - x_{F_2}\|^2,\label{eq:EG_is_non_cocoercive_tech1}\\
    \ell\langle x_{F_1} - y_{F_1}, x-y \rangle &\ge&  \|x_{F_1} - y_{F_1}\|^2,\label{eq:EG_is_non_cocoercive_tech2}\\
    \ell\langle x_{F_1} - y_{F_2}, x-y + \gamma_1 y_{F_1}\rangle &\ge&  \|x_{F_1} - y_{F_2}\|^2,\label{eq:EG_is_non_cocoercive_tech3}\\
    \ell\langle x_{F_2} - y_{F_1}, x - \gamma_1 x_{F_1} - y\rangle &\ge&  \|x_{F_2} - y_{F_1}\|^2,\label{eq:EG_is_non_cocoercive_tech4}\\
    \ell\langle x_{F_2} - y_{F_2}, x - \gamma_1 x_{F_1} - y + \gamma_1 y_{F_1}\rangle &\ge&  \|x_{F_2} - y_{F_2}\|^2, \label{eq:EG_is_non_cocoercive_tech5}\\
    \ell\langle y_{F_1} - y_{F_2}, \gamma_1 y_{F_1}\rangle &\ge&  \|y_{F_1} - y_{F_2}\|^2 \label{eq:EG_is_non_cocoercive_tech6}
    \end{eqnarray}
    in order to prove that there exists $\ell$-cocoercive $F$ such that $F(x) = x_{F_1}, F(y) = y_{F_1}, F(x-\gamma_1 x_{F_1}) = x_{F_2}, F(y - \gamma_1 y_{F_1}) = y_{F_2}$. Below we derive these inequalities for $x, y, x_{F_1}, y_{F_1}, x_{F_2}, y_{F_2}$ defined in \eqref{eq:EG_bad_example}.
    
    \noindent \textbf{Proof of \eqref{eq:EG_is_non_cocoercive_tech1}.} We have
    \begin{eqnarray*}
        \ell\langle x_{F_1} - x_{F_2}, \gamma_1 x_{F_1} \rangle - \|x_{F_1} - x_{F_2}\|^2 &=& \ell\left(-\frac{1}{2\gamma_1} + \frac{1-\gamma_1\ell}{2\gamma_1}\right)\cdot\frac{-\gamma_1}{2\gamma_1} - \left(-\frac{1}{2\gamma_1} + \frac{1-\gamma_1\ell}{2\gamma_1}\right)^2\\
        &=& \frac{\ell^2}{4} - \frac{\ell^2}{4} = 0.
    \end{eqnarray*}
    
    \noindent \textbf{Proof of \eqref{eq:EG_is_non_cocoercive_tech2}.} We have
    \begin{eqnarray*}
        \ell\langle x_{F_1} - y_{F_1}, x-y \rangle - \|x_{F_1} - y_{F_1}\|^2 &=& -\ell\left(-\frac{1}{2\gamma_1} + \frac{1-\gamma_1\ell}{2\gamma_1}\right) - \left(-\frac{1}{2\gamma_1} + \frac{1-\gamma_1\ell}{2\gamma_1}\right)^2\\
        &&\quad - \left(\frac{1}{2\gamma_1} - \frac{1+\gamma_1\ell}{2\gamma_1}\right)^2\\
        &=& \frac{\ell^2}{2} - \frac{\ell^2}{4} - \frac{\ell^2}{4} = 0.
    \end{eqnarray*}
    
    \noindent \textbf{Proof of \eqref{eq:EG_is_non_cocoercive_tech3}.} We have
    \begin{eqnarray*}
        \ell\langle x_{F_1} - y_{F_2}, x-y + \gamma_1 y_{F_1}\rangle -  \|x_{F_1} - y_{F_2}\|^2 &=& \ell\left(-\frac{1}{2\gamma_1} + \frac{1-\gamma_1\ell}{2\gamma_1}\right)\left(-1 + \frac{1-\gamma_1\ell}{2}\right)\\
        &&\quad + \ell\left(\frac{1}{2\gamma_1} - \frac{1-\gamma_1^2\ell^2}{2\gamma_1}\right)\cdot\frac{1+\gamma_1\ell}{2}\\
        &&\quad - \left(-\frac{1}{2\gamma_1} + \frac{1-\gamma_1\ell}{2\gamma_1}\right)^2 - \left(\frac{1}{2\gamma_1} - \frac{1-\gamma_1^2\ell^2}{2\gamma_1}\right)^2\\
        &=& \frac{\ell^2(1+\gamma_1\ell)}{4} + \frac{\gamma_1\ell^3(1+\gamma_1\ell)}{4} - \frac{\ell^2}{4} - \frac{\gamma_1^2\ell^4}{4}\\
        &=& \frac{\gamma_1\ell}{2} > 0.
    \end{eqnarray*}
    
    \noindent \textbf{Proof of \eqref{eq:EG_is_non_cocoercive_tech4}.} We have
    \begin{eqnarray*}
        \ell\langle x_{F_2} - y_{F_1}, x - \gamma_1 x_{F_1} - y\rangle -  \|x_{F_2} - y_{F_1}\|^2 &=& \ell\left(\frac{1}{2\gamma_1} - \frac{1+\gamma_1\ell}{2\gamma_1}\right)\cdot\frac{-\gamma_1}{2\gamma_1} - \left(\frac{1}{2\gamma_1} - \frac{1+\gamma_1\ell}{2\gamma_1}\right)^2\\
        &=& \frac{\ell^2}{4} - \frac{\ell^2}{4} = 0.
    \end{eqnarray*}
    
    \noindent \textbf{Proof of \eqref{eq:EG_is_non_cocoercive_tech5}.} We have
    \begin{eqnarray*}
        \ell\langle x_{F_2} - y_{F_2}, x - \gamma_1 x_{F_1} - y + \gamma_1 y_{F_1}\rangle -  \|x_{F_2} - y_{F_2}\|^2 &=& \ell\left(\frac{1}{2\gamma_1} - \frac{1-\gamma_1^2\ell^2}{2\gamma_1}\right)\left(-\frac{1}{2} + \frac{1+\gamma_1\ell}{2}\right)\\
        && - \left(\frac{1}{2\gamma_1} - \frac{1-\gamma_1^2\ell^2}{2\gamma_1}\right)^2\\
        &=& \frac{\gamma_1^2\ell^4}{4} - \frac{\gamma_1^2\ell^4}{4} = 0.
    \end{eqnarray*}
    
    \noindent \textbf{Proof of \eqref{eq:EG_is_non_cocoercive_tech6}.} We have
    \begin{eqnarray*}
        \ell\langle y_{F_1} - y_{F_2}, \gamma_1 y_{F_1}\rangle -  \|y_{F_1} - y_{F_2}\|^2 &=& \ell\left(\frac{1+\gamma_1\ell}{2\gamma_1} - \frac{1-\gamma_1^2\ell^2}{2\gamma_1}\right)\frac{1+\gamma_1\ell}{2} - \left(\frac{1+\gamma_1\ell}{2\gamma_1} - \frac{1-\gamma_1^2\ell^2}{2\gamma_1}\right)^2\\
        &=& \frac{\ell^2(1+\gamma_1\ell)^2}{4} - \frac{\ell^2(1+\gamma_1\ell)^2}{4} = 0.
    \end{eqnarray*}
    
    That is, inequalities \eqref{eq:EG_is_non_cocoercive_tech1}-\eqref{eq:EG_is_non_cocoercive_tech6} hold and, as a result, there exists $\ell$-cocoercive operator $F$ such that $F(x) = x_{F_1}, F(y) = y_{F_1}, F(x-\gamma_1 x_{F_1}) = x_{F_2}, F(y - \gamma_1 y_{F_1}) = y_{F_2}$. Finally, for any $\gamma_2 > 0$ we have
    \begin{eqnarray*}
        \|x - \gamma_2 F(x - \gamma_1 F(x)) - y + \gamma_2 F(y - \gamma_1 F(y))\|^2 &=& 1 + \gamma_2^2\left(\frac{1}{2\gamma_1} - \frac{1-\gamma_1^2\ell^2}{2\gamma_1}\right)^2\\
        &=& 1 + \frac{\gamma_1^2\gamma_2^2\ell^4}{4} > 1 = \|x-y\|^2.
    \end{eqnarray*}
    In view of Lemma~\ref{lem:l_non_exp_and_cocoercive}, it means that operator $F_{\algname{EG},\gamma_1} = F(\Id - \gamma_1 F)$ is non-cocoercive.
\end{proof}

We emphasize that in the example \eqref{eq:EG_bad_example} one can multiply all points by arbitrary $\alpha > 0$ and get $\|x-y\| = \alpha$: the proof will remain almost unchanged. That is, the points $x,y$ can be arbitrary close/far to each other in the example showing non-cocoercivity of \algname{EG} operator.

\subsection{Proof of Lemma~\ref{lem:EG_norm_is_non_incr}}\label{sec:PEP_proof_of_EG_norm_relax}
As we explain in Section~\ref{sec:EG}, we obtain the proof of Lemma~\ref{lem:EG_norm_is_non_incr} via solving the following problem:
\begin{eqnarray}
    \Delta_{\algname{EG}}(L,\gamma_1,\gamma_2) = &\max & \|F(x^1)\|^2 - \|F(x^0)\|^2 \notag\\
    &\text{s.t.} & F \text{ is monotone and $L$-Lipschitz},\;x^0\in\R^d,\notag\\
    && \|x^0 - x^*\|^2 \leq 1,\notag\\
    && x^{1} = x^0 - \gamma_2 F\left(x^0 - \gamma_1 F(x^0)\right)\notag
\end{eqnarray}
with $\gamma_1 = \gamma_2 = \gamma$. As for \eqref{eq:EG_norm_PEP}, we construct a corresponding SDP and solve it for different values of $L$ and $\gamma$. In these numerical tests, we observed that $\Delta_{\algname{EG}}(L,\gamma_1,\gamma_2) \approx 0$ for all tested pairs of $L$ and $\gamma$ and the dual variables $\lambda_1, \lambda_2, \lambda_3$ that correspond to the constraints
 \begin{eqnarray}
    0 &\le& \frac{1}{\gamma}\langle F(x^k) - F(x^{k+1}), x^k - x^{k+1} \rangle, \notag\\
    0 &\le& \frac{1}{\gamma}\langle F(x^k - \gamma F(x^k)) - F(x^{k+1}), x^k - \gamma F(x^k) - x^{k+1} \rangle, \notag\\
    \|F(x^k - \gamma F(x^k)) - F(x^{k+1})\|^2 &\le& L^2\|x^k - \gamma F(x^k) - x^{k+1}\|^2 \notag
\end{eqnarray}
are always close to the constants $2, \nicefrac{1}{2},$ and $\nicefrac{3}{2}$, while other dual variables are negligible. Although $\lambda_2$ and $\lambda_3$ were sometimes slightly smaller, e.g., sometimes we had $\lambda_2 \approx \nicefrac{3}{5}$ and $\lambda_3 \approx \nicefrac{13}{20}$, we simplified these dependencies and simply summed up the corresponding inequalities with weights $\lambda_1 = 2$, $\lambda_2 = \nicefrac{1}{2}$ and $\lambda_3 = \nicefrac{3}{2}$ respectively. After that it was just needed to rearrange the terms and apply Young's inequality to some inner products. 

The rigorous proof is provided below.
\begin{lemma}[Lemma~\ref{lem:EG_norm_is_non_incr}]\label{lem:EG_norm_is_non_incr_appendix}
    Let $F: \R^d \to \R^d$ be monotone and $L$-Lipschitz, $0 < \gamma \le \nicefrac{1}{\sqrt{2}L}$. Then for all $k\ge 0$ the iterates produced by \eqref{eq:EG_update} satisfy $\|F(x^{k+1})\| \le \|F(x^k)\|$.
\end{lemma}
\begin{proof}
    Since $F$ is monotone and $L$-Lipschitz we have
    \begin{eqnarray}
        0 &\le& \langle F(x^k) - F(x^{k+1}), x^k - x^{k+1} \rangle, \notag\\
        0 &\le& \langle F(x^k - \gamma F(x^k)) - F(x^{k+1}), x^k - \gamma F(x^k) - x^{k+1} \rangle, \notag\\
        \|F(x^k - \gamma F(x^k)) - F(x^{k+1})\|^2 &\le& L^2\|x^k - \gamma F(x^k) - x^{k+1}\|^2. \notag
    \end{eqnarray}
    Using the update rule of \eqref{eq:EG_update} and introducing new notation $\tx^k = x^k - \gamma F(x^k)$, we get
    \begin{eqnarray}
        0 &\le& \langle F(x^k) - F(x^{k+1}), F(\tx^k) \rangle, \notag\\
        0 &\le& \langle F(\tx^k) - F(x^{k+1}), F(\tx^k) - F(x^k) \rangle, \notag\\
        \|F(\tx^k) - F(x^{k+1})\|^2 &\le& L^2\gamma^2\|F(\tx^k) - F(x^k)\|^2. \notag
    \end{eqnarray}
    Summing up these inequalities with weights $\lambda_1 = 2$, $\lambda_2 = \nicefrac{1}{2}$ and $\lambda_3 = \nicefrac{3}{2}$ respectively, we derive
    \begin{eqnarray*}
        \frac{3}{2}\|F(\tx^k) - F(x^{k+1})\|^2 &\le& 2\langle F(x^k) - F(x^{k+1}), F(\tx^k) \rangle + \frac{1}{2}\langle F(\tx^k) - F(x^{k+1}), F(\tx^k) - F(x^k) \rangle\\
        &&\quad + \frac{3L^2\gamma^2}{2}\|F(\tx^k) - F(x^k)\|^2.
    \end{eqnarray*}
    Next, we expand the squared norms and rearrange the terms:
    \begin{eqnarray*}
        \frac{3}{2}\|F(x^{k+1})\|^2 &\le& \left(2 - \frac{1}{2} - 3L^2\gamma^2\right)\langle F(x^k), F(\tx^k) \rangle  + \left(-2 - \frac{1}{2} + 3\right)\langle F(x^{k+1}), F(\tx^k) \rangle\\
        &&\quad + \frac{1}{2}\langle F(x^{k+1}), F(x^k)\rangle + \left(\frac{1}{2} - \frac{3}{2} + \frac{3L^2\gamma^2}{2}\right)\|F(\tx^k)\|^2 + \frac{3L^2\gamma^2}{2}\|F(x^k)\|^2\\
        &=& \left(\frac{3}{2} - 3L^2\gamma^2\right)\langle F(x^k), F(\tx^k) \rangle  + \frac{1}{2}\langle F(x^{k+1}), F(\tx^k) \rangle + \frac{1}{2}\langle F(x^{k+1}), F(x^k)\rangle\\
        &&\quad + \left(\frac{3L^2\gamma^2}{2} - 1\right)\|F(\tx^k)\|^2 + \frac{3L^2\gamma^2}{2}\|F(x^k)\|^2.
    \end{eqnarray*}
    We notice that $\frac{3}{2} - 3L^2\gamma^2 \ge 0$ since $\gamma \le \frac{1}{\sqrt{2}L}$. Therefore, applying Young's inequality \eqref{eq:young} to upper bound the inner products, we derive
    \begin{eqnarray*}
        \frac{3}{2}\|F(x^{k+1})\|^2 &\le& \left(\frac{3}{4} - \frac{3L^2\gamma^2}{2}\right)\left(\|F(x^k)\|^2 + \|F(\tx^k)\|^2\right)  + \frac{1}{4}\left( \|F(x^{k+1})\|^2 + \|F(\tx^k)\|^2 \right)\\
        &&\quad + \frac{1}{4}\left( \|F(x^{k+1})\|^2 + \|F(x^k)\|^2\right) + \left(\frac{3L^2\gamma^2}{2} - 1\right)\|F(\tx^k)\|^2 + \frac{3L^2\gamma^2}{2}\|F(x^k)\|^2\\
        &=& \|F(x^k)\|^2 + \frac{1}{2}\|F(x^{k+1})\|^2.
    \end{eqnarray*}
    Rearranging the terms, we get the result.
\end{proof}

\subsection{Proof of Theorem~\ref{thm:EG_last_iter_conv_non_linear}}

\begin{theorem}[Theorem~\ref{thm:EG_last_iter_conv_non_linear}; Last-iterate convergence of \eqref{eq:EG_update}: non-linear case]\label{thm:EG_last_iter_conv_non_linear_appendix}
    Let $F: \R^d \to \R^d$ be monotone and $L$-Lipschitz. Then for all $K\ge 0$
    \begin{equation}
        \|F(x^K)\|^2 \le \frac{\|x^0 - x^*\|^2}{\gamma^2(1-L^2\gamma^2)(K+1)}, \label{eq:EG_last_iter_conv_non_linear_appendix}
    \end{equation}
    where $x^K$ is produced by \eqref{eq:EG_update} with stepsize $0 < \gamma \le \nicefrac{1}{\sqrt{2}L}$. Moreover,
    \begin{equation}
        \texttt{Gap}_F(x^K) = \max_{y\in\R^d:\|y-x^*\| \le \|x^0 - x^*\|}\langle F(y), x^K - y \rangle \le \frac{2\|x^0 - x^*\|^2}{\gamma\sqrt{1-L^2\gamma^2}\sqrt{K+1}}. \label{eq:EG_last_iter_conv_non_linear_gap_appendix}
    \end{equation}
\end{theorem}
\begin{proof}
    We notice that in the proof of Lemma~\ref{lem:cocoercivity_of_EG_linear_non_spec} we get
    \begin{eqnarray}
        \|\hx - \hy\|^2 &\overset{\eqref{eq:EG_non_exp_non_spectral_technical_2}}{\le}& \|x-y\|^2 + \gamma^2\left(\|F(\tx) - F(\ty) - F(x) + F(y)\|^2 - \|F(x) - F(y)\|^2\right) \notag
    \end{eqnarray}
    without using linearity of $F$. Here, $x$ and $y$ are arbitrary points in $\R^d$ and
    \begin{equation*}
        \tx = x - \gamma F(x),\quad \ty = y - \gamma F(y),\quad \hx = x - \gamma F(\tx),\quad \hy = y - \gamma F(\ty).
    \end{equation*}
    Taking $y = x^*$ and $x = x^k$ we get $\ty = \hy = x^*$, $\hx = x^{k+1}$, and
    \begin{eqnarray*}
        \|x^{k+1} - x^*\|^2 &\le& \|x^k - x^*\|^2 + \gamma^2\|F(x^k-\gamma F(x^k)) - F(x^k)\| - \gamma^2\|F(x^k)\|^2\\
        &\overset{\eqref{eq:L_lip_def}}{\le}& \|x^k-x^*\|^2 + L^2\gamma^4\|F(x^k)\|^2 - \gamma^2\|F(x^k)\|^2 
    \end{eqnarray*}
    implying
    \begin{equation}
        \gamma^2(1-L^2\gamma^2)\|F(x^k)\|^2 \le \|x^k - x^*\|^2 - \|x^{k+1} - x^*\|^2.\label{eq:EG_non_increasing_distance}
    \end{equation}
    Summing up these inequalities for $k=0,1,\ldots, K$ and dividing the result by $\gamma^2(1-L^2\gamma^2)(K+1)$, we obtain
    \begin{eqnarray*}
        \frac{1}{K+1}\sum\limits_{k=0}^{K}\|F(x^k)\|^2 &\le& \frac{1}{\gamma^2(1-L^2\gamma^2)(K+1)}\sum\limits_{k=0}^{K}\left(\|x^k - x^*\|^2 - \|x^{k+1} - x^*\|^2\right)\\
        &=& \frac{\|x^0 - x^*\|^2 - \|x^{K+1} - x^*\|^2}{\gamma^2(1-L^2\gamma^2)(K+1)}\\
        &\le& \frac{\|x^0 - x^*\|^2}{\gamma^2(1-L^2\gamma^2)(K+1)}.
    \end{eqnarray*}
    Next, applying Lemma~\ref{lem:EG_norm_is_non_incr}, we conclude
    \begin{equation*}
        \|F(x^K)\|^2 \le \frac{1}{K+1}\sum\limits_{k=0}^{K}\|F(x^k)\|^2 \le \frac{\|x^0 - x^*\|^2}{\gamma^2(1-L^2\gamma^2)(K+1)},
    \end{equation*}
    which gives \eqref{eq:EG_last_iter_conv_non_linear_appendix}. Finally, we notice that $\|x^{k+1} - x^*\| \le \|x^k - x^*\|$, which can be seen from \eqref{eq:EG_non_increasing_distance}. Therefore, using monotonicity of $F$ and Cauchy-Schwarz inequality, we derive
    \begin{eqnarray*}
        \texttt{Gap}_F(x^K) &=& \max_{y\in\R^d:\|y-x^*\| \le \|x^0 - x^*\|}\langle F(y), x^K - y \rangle\\
        &\overset{\eqref{eq:monotonicity_def}}{\leq}& \max_{y\in\R^d:\|y-x^*\| \le \|x^0 - x^*\|}\langle F(x^K), x^K - y \rangle\\
        &\leq& \|F(x^K)\|\cdot\max_{y\in\R^d:\|y-x^*\| \le \|x^0 - x^*\|} \|x^K - y\|\\ 
        &\overset{\eqref{eq:EG_non_increasing_distance}}{\leq}& 2\|F(x^K)\|\cdot\|x^0 - x^*\|\\
        &\overset{\eqref{eq:EG_last_iter_conv_non_linear_appendix}}{\leq}& \frac{2\|x^0 - x^*\|^2}{\gamma\sqrt{1-L^2\gamma^2}\sqrt{K+1}},
    \end{eqnarray*}
    which finishes the proof.
\end{proof}

\newpage

\section{MISSING PROOFS FROM SECTION~\ref{sec:OG}}

\subsection{Proof of Theorem~\ref{thm:F_OG_is_not_star_cocoercive}}

For convenience, we derive non-cocoercivity of $F_{\algname{OG},\gamma}$ and $F_{\algname{EFTP},\gamma}$ separately.

\subsubsection{Non-Cocoercivity of $F_{\algname{OG},\gamma}$}

Before we provide the proof, we state the following technical lemma.

\begin{lemma}\label{lem:F_OG_linearity}
    Let operator $F: \R^d \to \R^d$ be linear, $x^*$ be such that $F(x^*) = 0$, and $\gamma > 0$. Then, operator $F_{\algname{OG},\gamma}$ is linear and $F_{\algname{OG},\gamma}(z^*) = 0$ for $z^* = ((x^*)^\top, (x^*)^\top)^\top$.
\end{lemma}
\begin{proof}
    We start with proving linearity. Consider arbitrary
    \begin{eqnarray*}
        \alpha,\beta\in\R, \quad x,y,x',y'\in\R^d,\quad z = \begin{pmatrix}x\\ y \end{pmatrix}, z' = \begin{pmatrix}x'\\ y' \end{pmatrix} \in \R^{2d}.
    \end{eqnarray*}
    Then
    \begin{eqnarray*}
        F_{\algname{OG},\gamma}(\alpha z + \beta z') &=& \begin{pmatrix} 2F & -F \\ -\frac{1}{\gamma}\Id & \frac{1}{\gamma}\Id \end{pmatrix} \begin{pmatrix}\alpha x + \beta x' \\ \alpha y + \beta y' \end{pmatrix} = \begin{pmatrix}2F(\alpha x + \beta x') - F(\alpha y + \beta y') \\ -\frac{1}{\gamma}(\alpha x + \beta x') + \frac{1}{\gamma}(\alpha y + \beta y') \end{pmatrix}\\
        &=& \alpha \begin{pmatrix}2F(x) - F(y) \\ -\frac{1}{\gamma}x + \frac{1}{\gamma}y \end{pmatrix} + \beta \begin{pmatrix}2F(x') - F(y') \\ -\frac{1}{\gamma}x' + \frac{1}{\gamma}y' \end{pmatrix}\\
        &=& \alpha\begin{pmatrix} 2F & -F \\ -\frac{1}{\gamma}\Id & \frac{1}{\gamma}\Id \end{pmatrix} \begin{pmatrix}x \\ y\end{pmatrix} + \beta\begin{pmatrix} 2F & -F \\ -\frac{1}{\gamma}\Id & \frac{1}{\gamma}\Id \end{pmatrix} \begin{pmatrix}x' \\ y'\end{pmatrix}\\
        &=& \alpha F_{\algname{OG},\gamma}(z) + \beta F_{\algname{OG},\gamma}(z'),
    \end{eqnarray*}
    i.e., $F_{\algname{OG},\gamma}$ is linear. Next, let $F(x^*) = 0$ for some $x^*$. For
    \begin{equation*}
        z^* = \begin{pmatrix} x^* \\ x^* \end{pmatrix}
    \end{equation*}
    we derive that $F_{\algname{OG},\gamma}(z^*) = 0$:
    \begin{eqnarray*}
        F_{\algname{OG},\gamma}(z^*) = \begin{pmatrix} 2F & -F \\ -\frac{1}{\gamma}\Id & \frac{1}{\gamma}\Id \end{pmatrix} \begin{pmatrix}x^* \\ x^* \end{pmatrix} = \begin{pmatrix}2F(x^*) - F(x^*) \\ -\frac{1}{\gamma}x^* + \frac{1}{\gamma}x^* \end{pmatrix} = \begin{pmatrix}0 \\ 0 \end{pmatrix} = 0.
    \end{eqnarray*}
\end{proof}

Using this lemma, we establish the following result.
\begin{theorem}[Non-cocoercivity of $F_{\algname{OG},\gamma}$]\label{thm:F_OG_is_not_star_cocoercive_appendix}
    Let the linear operator $F(x) = \mA x$ be monotone and $L$-Lipschitz. Assume that $\Sp(\nabla F(x)) = \Sp(\mA)$ contains at least one eigenvalue $\hat{\lambda}$ such that $\Re(\hat{\lambda}) = 0$ and $\Im(\hat{\lambda}) \neq 0$.  Then, for any $\ell > 0$ and $\gamma > 0$ operator $F_{\algname{OG},\gamma}$ is not $\ell$-star-cocoercive around $x^*$.
\end{theorem}
\begin{proof}
    In view of Lemma~\ref{lem:cocoer_equiv_star_cocoer_for_lin_ops}, it is sufficient to show that $F_{\algname{OG},\gamma}$ is not $\ell$-cocoercive for any positive $\ell, \gamma > 0$. Since $\Sp(\mA)$ contains $\hat{\lambda}$ with $\Re(\hat{\lambda}) = 0$ and $\Im(\hat{\lambda}) \neq 0$, $\Sp(\mA)$ is not contained in any disk centered in $\nicefrac{\ell}{2}$ and of radius $\nicefrac{\ell}{2}$. Therefore, due to Lemma~\ref{prop:spectum_ell} operator $F$ is not $\ell$-cocoercive for any $\ell$. Let us fix arbitrary $\ell > 0$ and $\gamma > 0$. There exist points $x, x'$ such that
    \begin{equation}
        \|F(x) - F(x')\|^2 > \frac{\ell}{2}\langle x - x', F(x) - F(x') \rangle \label{eq:F_OG_is_not_star_cocoercive_tech_1}
    \end{equation}
    Let us show that $F_{\algname{OG},\gamma}$ is not $\ell$-cocoercive. In view of Lemma~\ref{lem:l_non_exp_and_cocoercive}, it is sufficient to show that $\Id - \frac{2}{\ell}F_{\algname{OG},\gamma}$ is not non-expansive. Consider the following points:
    \begin{equation*}
        z = \begin{pmatrix}x\\ x^*\end{pmatrix},\quad \hat{z} = z - \frac{2}{\ell}F_{\algname{OG},\gamma}(z),\quad z' = \begin{pmatrix}x'\\ x^*\end{pmatrix},\quad \hat{z}' = z' - \frac{2}{\ell}F_{\algname{OG},\gamma}(z').
    \end{equation*}
    Then
    \begin{eqnarray*}
        \|\hat z - \hat z'\|^2 &=& \left\|z - \frac{2}{\ell}F_{\algname{OG},\gamma}(z) - z' + \frac{2}{\ell}F_{\algname{OG},\gamma}(z')\right\|^2\\
        &=& \left\|\begin{pmatrix}x\\ x^* \end{pmatrix} - \frac{2}{\ell}\begin{pmatrix} 2F & -F \\ -\frac{1}{\gamma}\Id & \frac{1}{\gamma}\Id \end{pmatrix}\begin{pmatrix}x\\ x^* \end{pmatrix} - \begin{pmatrix}x'\\ x^* \end{pmatrix} + \frac{2}{\ell}\begin{pmatrix} 2F & -F \\ -\frac{1}{\gamma}\Id & \frac{1}{\gamma}\Id \end{pmatrix}\begin{pmatrix}x'\\ x^* \end{pmatrix}\right\|^2\\
        &=& \left\|\begin{pmatrix}x - \frac{4}{\ell}F(x) - x' + \frac{4}{\ell} F(x')\\ x^* + \frac{2}{\ell\gamma}x - \frac{2}{\ell\gamma}x^* - x^* - \frac{2}{\ell\gamma}x' + \frac{2}{\ell\gamma}x^* \end{pmatrix}\right\|^2\\
        &=& \left\|x - x' - \frac{4}{\ell}\left(F(x) - F(x')\right)\right\|^2 + \frac{4}{\ell^2\gamma^2}\|x-x'\|^2 \\
        &=& \left(1 + \frac{4}{\ell^2\gamma^2}\right)\|x-x'\|^2  + \frac{16}{\ell^2}\left(\|F(x) - F(x')\|^2 - \frac{\ell}{2}\langle x - x', F(x) - F(x') \rangle\right)\\
        &\overset{\eqref{eq:F_OG_is_not_star_cocoercive_tech_1}}{>}& \left(1 + \frac{4}{\ell^2\gamma^2}\right)\|x-x'\|^2 > \|x - x'\|^2 = \|z - z'\|^2, 
    \end{eqnarray*}
    i.e., $\Id - \frac{2}{\ell}F_{\algname{OG},\gamma}$ is not non-expansive.
\end{proof}

\subsubsection{Non-Cocoercivity of $F_{\algname{EFTP},\gamma}$}

First of all, for any
\begin{equation*}
    z = \begin{pmatrix}x \\ y \end{pmatrix}
\end{equation*}
one can rewrite $F_{\algname{EFTP}, \gamma}(z)$ as
\begin{equation*}
    F_{\algname{EFTP}, \gamma}(z) = \begin{pmatrix}F\left(x - \gamma F(y)\right)\\ \frac{1}{\gamma}(y - x) + F(y) \end{pmatrix}.
\end{equation*}
Using this, we derive the following technical result.
\begin{lemma}\label{lem:F_EFTP_linearity}
    Let operator $F: \R^d \to \R^d$ be linear, $x^*$ be such that $F(x^*) = 0$, and $\gamma > 0$. Then, operator $F_{\algname{EFTP},\gamma}$ is linear and $F_{\algname{EFTP},\gamma}(z^*) = 0$ for $z^* = ((x^*)^\top, (x^*)^\top)^\top$.
\end{lemma}
\begin{proof}
    We start with proving linearity. Consider arbitrary
    \begin{eqnarray*}
        \alpha,\beta\in\R, \quad x,y,x',y'\in\R^d,\quad z = \begin{pmatrix}x\\ y \end{pmatrix}, z' = \begin{pmatrix}x'\\ y' \end{pmatrix} \in \R^{2d}.
    \end{eqnarray*}
    Then
    \begin{eqnarray*}
        F_{\algname{EFTP},\gamma}(\alpha z + \beta z') &=& \begin{pmatrix}F\left(\alpha x + \beta x' - \gamma F(\alpha y + \beta y')\right)\\ \frac{1}{\gamma}\left(\alpha y + \beta y' - \alpha x - \beta x'\right) + F(\alpha y + \beta y') \end{pmatrix}\\
        &=& \begin{pmatrix}F\left(\alpha x - \gamma \alpha F(y) + \beta x' - \gamma \beta F(y')\right)\\ \frac{1}{\gamma}\left(\alpha y  - \alpha x\right) + \alpha F(y) + \frac{1}{\gamma}\left(\beta y'  - \beta x'\right) + \beta F(y') \end{pmatrix}\\
        &=& \alpha \begin{pmatrix}F\left(x - \gamma F(y)\right)\\ \frac{1}{\gamma}(y - x) + F(y) \end{pmatrix} + \beta \begin{pmatrix}F\left(x' - \gamma F(y')\right)\\ \frac{1}{\gamma}(y' - x') + F(y') \end{pmatrix}\\
        &=& \alpha F_{\algname{EFTP},\gamma}(z) + \beta F_{\algname{EFTP},\gamma}(z'),
    \end{eqnarray*}
    i.e., $F_{\algname{EFTP},\gamma}$ is linear. Next, let $F(x^*) = 0$ for some $x^*$. For
    \begin{equation*}
        z^* = \begin{pmatrix} x^* \\ x^* \end{pmatrix}
    \end{equation*}
    we derive that $F_{\algname{EFTP},\gamma}(z^*) = 0$:
    \begin{eqnarray*}
        F_{\algname{OG},\gamma}(z^*) = \begin{pmatrix}F\left(x^* - \gamma F(x^*)\right)\\ \frac{1}{\gamma}(x^* - x^*) + F(y^*) \end{pmatrix} = \begin{pmatrix}F(x^*)\\ 0 \end{pmatrix} = \begin{pmatrix}0\\ 0 \end{pmatrix} = 0.
    \end{eqnarray*}
\end{proof}

Using this lemma, we establish the following result.
\begin{theorem}[Non-cocoercivity of $F_{\algname{EFTP},\gamma}$]\label{thm:F_EFTP_is_not_star_cocoercive_appendix}
    Let the linear operator $F(x) = \mA x$ be monotone and $L$-Lipschitz. Assume that $\Sp(\nabla F(x)) = \Sp(\mA)$ contains at least one eigenvalue $\hat{\lambda}$ such that $\Re(\hat{\lambda}) = 0$ and $\Im(\hat{\lambda}) \neq 0$.  Then, for any $\ell > 0$ and $\gamma > 0$ operator $F_{\algname{EFTP},\gamma}$ is not $\ell$-star-cocoercive around $x^*$.
\end{theorem}
\begin{proof}
    In view of Lemma~\ref{lem:cocoer_equiv_star_cocoer_for_lin_ops}, it is sufficient to show that $F_{\algname{EFTP},\gamma}$ is not $\ell$-cocoercive for any positive $\ell, \gamma > 0$. Since $\Sp(\mA)$ contains $\hat{\lambda}$ with $\Re(\hat{\lambda}) = 0$ and $\Im(\hat{\lambda}) \neq 0$, $\Sp(\mA)$ is not contained in any disk centered in $\nicefrac{\ell}{2}$ and of radius $\nicefrac{\ell}{2}$. Therefore, due to Lemma~\ref{prop:spectum_ell} operator $F$ is not $\ell$-cocoercive for any $\ell$. Let us fix arbitrary $\ell > 0$ and $\gamma > 0$. There exist points $x, x'$ such that
    \begin{equation}
        \|F(x) - F(x')\|^2 > \ell\langle x - x', F(x) - F(x') \rangle \label{eq:F_EFTP_is_not_star_cocoercive_tech_1}
    \end{equation}
    Let us show that $F_{\algname{EFTP},\gamma}$ is not $\ell$-cocoercive. In view of Lemma~\ref{lem:l_non_exp_and_cocoercive}, it is sufficient to show that $\Id - \frac{2}{\ell}F_{\algname{EFTP},\gamma}$ is not non-expansive. Consider the following points:
    \begin{equation*}
        z = \begin{pmatrix}x\\ x^*\end{pmatrix},\quad \hat{z} = z - \frac{2}{\ell}F_{\algname{EFTP},\gamma}(z),\quad z' = \begin{pmatrix}x'\\ x^*\end{pmatrix},\quad \hat{z}' = z' - \frac{2}{\ell}F_{\algname{EFTP},\gamma}(z').
    \end{equation*}
    Then
    \begin{eqnarray*}
        \|\hat z - \hat z'\|^2 &=& \left\|z - \frac{2}{\ell}F_{\algname{EFTP},\gamma}(z) - z' + \frac{2}{\ell}F_{\algname{EFTP},\gamma}(z')\right\|^2\\
        &=& \left\|\begin{pmatrix}x\\ x^* \end{pmatrix} - \frac{2}{\ell}\begin{pmatrix} F(x) \\ -\frac{1}{\gamma}\left(x^* - x\right) \end{pmatrix} - \begin{pmatrix}x'\\ x^* \end{pmatrix} + \frac{2}{\ell}\begin{pmatrix} F(x') \\ -\frac{1}{\gamma}\left(x^* - x'\right) \end{pmatrix}\right\|^2\\
        &=& \left\|\begin{pmatrix}x - \frac{2}{\ell}F(x) - x' + \frac{2}{\ell} F(x')\\ x^* - \frac{2}{\ell\gamma}x + \frac{2}{\ell\gamma}x^* - x^* + \frac{2}{\ell\gamma}x' - \frac{2}{\ell\gamma}x^* \end{pmatrix}\right\|^2\\
        &=& \left\|x - x' - \frac{2}{\ell}\left(F(x) - F(x')\right)\right\|^2 + \frac{4}{\ell^2\gamma^2}\|x-x'\|^2 \\
        &=& \left(1 + \frac{4}{\ell^2\gamma^2}\right)\|x-x'\|^2  + \frac{4}{\ell^2}\left(\|F(x) - F(x')\|^2 - \ell\langle x - x', F(x) - F(x') \rangle\right)\\
        &\overset{\eqref{eq:F_EFTP_is_not_star_cocoercive_tech_1}}{>}& \left(1 + \frac{4}{\ell^2\gamma^2}\right)\|x-x'\|^2 > \|x - x'\|^2 = \|z - z'\|^2, 
    \end{eqnarray*}
    i.e., $\Id - \frac{2}{\ell}F_{\algname{EFTP},\gamma}$ is not non-expansive.
\end{proof}

\subsection{Random-Iterate Convergence of \eqref{eq:EFTP_update} for Star-Monotone Operators}\label{sec:rand_iter_EFTP}
\begin{theorem}[Random-iterate convergence of \eqref{eq:EFTP_update}]\label{thm:EFTP_random_itre_convergence}
     Let $F:\R^d \to \R^d$ be star-monotone around $x^*$, i.e., $F(x^*) = 0$ and
    \begin{equation}
        \forall x\in\R^d \quad \langle F(x), x - x^* \rangle \ge 0, \notag
    \end{equation}
    and $L$-Lipschitz. Then for all $K \ge 0$ we have
    \begin{equation}
        \EE\|F(\widehat{x}^{K})\|^2 \le \frac{\|x^0 - x^*\|^2}{\gamma^2(1-10\gamma^2 L^2)(K+1)}, \label{eq:random_iter_conv_EFTP}
    \end{equation}
    where $\widehat{x}^{K}$ is chosen uniformly at random from the set of iterates $\{\tx^0, \tx^1, \ldots, \tx^K\}$ produced by \eqref{eq:EFTP_update} with $0 < \gamma < \nicefrac{1}{\sqrt{10}L}$.
\end{theorem}
\begin{proof}
    Lemma~\ref{lem:EFTP_aux_lemma} with $x = x^*$ implies 
    \begin{equation*}
        2\gamma \langle F(\tx^{k+1}), \tx^{k+1} - x^* \rangle \le \|x^k - x^*\|^2 - \|x^{k+1} - x^*\|^2 - \|\tx^{k+1} - x^{k}\|^2 + \gamma^2 L^2 \|\tx^k - \tx^{k+1}\|^2.
    \end{equation*}
    Since $F$ is star monotone and $\tx^{k+1} - x^{k} = -\gamma F(\tx^k)$, we have
    \begin{eqnarray*}
        \gamma^2\|F(\tx^k)\|^2 &\le& \|x^k - x^*\|^2 - \|x^{k+1} - x^*\|^2 + \gamma^2 L^2 \|\tx^k - \tx^{k+1}\|^2\\
        &\overset{\eqref{eq:EFTP_update_to_OG_update}}{=}& \|x^k - x^*\|^2 - \|x^{k+1} - x^*\|^2 + \gamma^4 L^2 \|2F(\tx^k) - F(\tx^{k-1})\|^2\\
        &\overset{\eqref{eq:a+b}}{\le}& \|x^k - x^*\|^2 - \|x^{k+1} - x^*\|^2 + \gamma^4 L^2 \left(1+\frac{1}{4}\right) \|2F(\tx^k)\|^2\\
        &&\quad + \gamma^4 L^2 (1+4) \|F(\tx^{k-1})\|^2\\
        &=& \|x^k - x^*\|^2 - \|x^{k+1} - x^*\|^2 + 5\gamma^4 L^2 \|F(\tx^k)\|^2 + 5\gamma^4 L^2 \|F(\tx^{k-1})\|^2
    \end{eqnarray*}
    for $k \ge 1$, and
    \begin{eqnarray*}
        \gamma^2\|F(\tx^0)\|^2 &\le& \|x^0 - x^*\|^2 - \|x^{1} - x^*\|^2 + \gamma^2 L^2 \|\tx^0 - \tx^{1}\|^2\\
        &\overset{x^0 = \tx^0}{=}& \|x^0 - x^*\|^2 - \|x^{1} - x^*\|^2 + \gamma^4 L^2 \|F(\tx^{0})\|^2\\
        &\le& \|x^0 - x^*\|^2 - \|x^{1} - x^*\|^2 + 5\gamma^4 L^2 \|F(\tx^{0})\|^2.
    \end{eqnarray*}
    Rearranging the terms, we derive for all $k\ge 1$ that
    \begin{eqnarray}
        \gamma^2(1-5\gamma^2L^2)\|F(\tx^k)\|^2 &\le& \|x^k - x^*\|^2 - \|x^{k+1} - x^*\|^2 + 5\gamma^4 L^2 \|F(\tx^{k-1})\|^2 \label{eq:EFTP_rand_it_technical_1},\\
        \gamma^2(1-5\gamma^2L^2)\|F(\tx^0)\|^2 &\le& \|x^0 - x^*\|^2 - \|x^{1} - x^*\|^2. \label{eq:EFTP_rand_it_technical_2}
    \end{eqnarray}
    Next, we sum up inequalities \eqref{eq:EFTP_rand_it_technical_1} for $k = 1,\ldots, K$ and \eqref{eq:EFTP_rand_it_technical_2}:
    \begin{eqnarray*}
        \gamma^2(1-5\gamma^2 L^2)\sum\limits_{k=0}^K\|F(\tx^k)\|^2 &\le& \sum\limits_{k=1}^K\left(\|x^k - x^*\|^2 - \|x^{k+1} - x^*\|^2\right) + 5\gamma^4L^2\sum\limits_{k=0}^{K-1}\|F(\tx^k)\|^2\\
        &\le& \|x^0 - x^*\|^2 + 5\gamma^4L^2\sum\limits_{k=0}^{K}\|F(\tx^k)\|^2.
    \end{eqnarray*}
    Rearranging the terms and dividing the result by $\gamma^2(1-10\gamma^2 L^2)(K+1)$, we get
    \begin{equation*}
        \frac{1}{K+1}\sum\limits_{k=0}^K\|F(\tx^k)\|^2 \le \frac{\|x^0 - x^*\|^2}{\gamma^2(1-10\gamma^2 L^2)(K+1)}.
    \end{equation*}
    Finally, since $\widehat{x}^K$ is chosen uniformly at random from the set $\{\tx^0,\tx^1,\ldots,\tx^K\}$ we derive
    \begin{equation*}
        \EE\|F(\widehat x^K)\|^2 = \frac{1}{K+1}\sum\limits_{k=0}^{K}\|F(\tx^k)\|^2 \le \frac{\|x^0 - x^*\|^2}{\gamma^2(1-10\gamma^2 L^2)(K+1)}.
    \end{equation*}
\end{proof}

\newpage

\section{HAMILTONIAN GRADIENT METHOD}\label{sec:HGM}
Although Hamiltonian Gradient Method (\algname{HGM}) is not considered as an approximation of \algname{PP}, it attracted a lot of attention during the recent few years. Therefore, it is worth to study whether the operator of \algname{HGM} is cocoercive. First of all, \algname{HGM}
\begin{equation}
\label{DetHamiltonian}
x^{k+1}=  x^k - \gamma \nabla F(x^k)^\top F(x^k), \tag{HGM}
\end{equation}
can be seen as \algname{GD} applied to minimize function $\cH(x)= \nicefrac{\|F(x)\|^2}{2}$.
The corresponding operator is $F_{\cH}(x) = \nabla F(x)^\top F(x)$.

\subsection{Affine Case}
Let us start with the affine case.
 
\begin{lemma}\label{lem:HGM_is_cocoercive_in_affine_case}
    Let $F(x) = \mA x + b$ be $L$-Lipschitz. Then, Hamiltonian operator $F_{\cH}(x) = \nabla F(x)^\top F(x)$ is $L$-cocoercive.
\end{lemma}
\begin{proof}
    We have
    \begin{eqnarray*}
        \cH(x) &=& \frac{1}{2}\|F(x)\|^2 \\
        &=& \frac{1}{2}\|\mA x + b\|^2\\
        &=& \frac{1}{2}x^\top \mA^\top \mA x + b^\top \mA x + \frac{1}{2}\|b\|^2.
    \end{eqnarray*}
    Since $\nabla^2 \cH(x) = \mA^\top \mA \succeq 0$ function $\cH(x)$ is convex. Next, $L$-Lipschitzness of $\mA$ implies that $\|\mA\|_2 = \sqrt{\lambda_{\max}(\mA^\top \mA)} \le L$, i.e., $\nabla^2 \cH(x) = \mA^\top \mA \preceq L\mI$. Therefore, $\cH(x)$ is $L$-smooth function. It is well known \citep{nesterov2018lectures} that the gradient of convex $L$-smooth function is $L$-cocoercive, i.e., $\nabla \cH(x)  = \nabla F(x)^\top F(x) = F_{\cH}(x)$ is $L$-cocoercive operator.
\end{proof}

As a direct application of Theorem~\ref{thm:last_iter_conv_GD} we get the following result.
\begin{theorem}[Last-iterate convergence of \eqref{DetHamiltonian}: affine case]
    Let $F(x) = \mA x + b$ be $L$-Lipschitz. Then for all $K\ge 0$ we have
    \begin{equation}
        \left\|\nabla F(x^K)^\top F(x^k)\right\|^2  \le \frac{L\|x^0 - x^*\|^2}{\gamma (K+1)}, \label{eq:last_iter_conv_HGM_affine_bad_guarantee}
    \end{equation}
    where $x^K$ is produced by \eqref{DetHamiltonian} with $0 < \gamma \le \nicefrac{1}{L}$.
\end{theorem}

However, this theorem completely ignores the fact that $\nabla F(x)^\top F(x)$ corresponds to the gradient of function $\cH(x)$. Taking into account that $\cH(x) = \frac{1}{2}\|\mA x + b\|^2$, one can prove that $\cH(x)$ is quasi-strongly convex \citep{necoara2019linear} and get the following result for Gradient Descent applied to minimize function $\cH(x)$.

\begin{theorem}[See Theorem 11 from \citet{necoara2019linear}]
    Let $F(x) = \mA x + b$. Then for all $K\ge 0$ we have
    \begin{equation}
        \|x^K - x^*\|^2 \le \left(\frac{1 -\kappa(\mA)}{1+\kappa(\mA)}\right)\|x^0 - x^*\|^2, \label{eq:last_iter_conv_HGM_affine_good_guarantee_func_val}
    \end{equation}
    where $\kappa(\mA) = \nicefrac{\sigma_{\min}^2(\mA)}{\sigma_{\max}^2(\mA)}$, $\sigma_{\min}^2(\mA)$ and $\sigma_{\max}^2(\mA)$ are the smallest non-zero and the largest singular values of $\mA$ respectively, and $x^K$ is produced by \eqref{DetHamiltonian} with $\gamma = \nicefrac{1}{\sigma_{\max}^2(\mA)}$.
\end{theorem}

Similar results are also derived in \citet{abernethy2019last,loizou2020stochastic}.

\subsection{General Case}
Next, we consider the setup when $F$ is monotone, $L$-Lipschitz, but not necessarily affine. In this case, it turns out that Hamiltonian operator $F_{\cH}$ can be non-cocoercive and function $\cH$ can be non-convex. To prove this, we provide an example of convex smooth function $f(x)$ such that $\|\nabla f(x)\|^2$ is non-convex.

\begin{theorem}[Non-cocoercivity of the Hamiltonian operator]\label{thm:HGM_non_cocoercive}
    Consider strongly convex smooth function $f(x) = \ln(1 + e^x) + \frac{x^2}{200}$ of a scalar argument $x \in \R$. Then, Hamiltonian operator $\cH(x) = \nabla F(x)^\top F(x)$ is non-cocoercive for monotone Lipschitz $F(x) = \nabla f(x)$. 
\end{theorem}
\begin{proof}
    Function $f(x) = \ln(1 + e^x) + \frac{x^2}{200}$ is logistic loss with $\ell_2$-regularization. Therefore, it is strongly convex smooth function and its gradient
    \begin{equation*}
        F(x) = \nabla f(x) = \frac{e^x}{1+e^x} + \frac{x}{100}
    \end{equation*}
    is (strongly) monotone and Lipschitz operator. Below we prove that Hamiltonian function $\cH(x) = \frac{1}{2}\|F(x)\|^2$ is non-convex. To show that we compute its second derivative:
    \begin{eqnarray*}
        2\cH(x) &=& \left(\frac{e^x}{1+e^x} + \frac{x}{100}\right)^2 = \frac{e^{2x}}{(1+e^x)^2} + \frac{xe^x}{50(1+e^x)} + \frac{x^2}{10000},\\
        2\nabla\cH(x) &=& \frac{2e^{2x}}{(1+e^x)^2} - \frac{2e^{3x}}{(1+e^x)^3} + \frac{(x+1)e^x}{50(1+e^x)} - \frac{x e^{2x}}{50(1+e^x)^2} + \frac{x}{5000}\\
        &=& \frac{2 e^{2x}}{(1+e^x)^3} + \frac{e^{2x} + e^x(x+1)}{50(1+e^x)^2} + \frac{x}{5000}\\
        2\nabla^2\cH(x) &=& \frac{4e^{2x}}{(1+e^x)^3} - \frac{6e^{3x}}{(1+e^x)^4} + \frac{2e^{2x} + e^x(x+2)}{50(1+e^x)^2} - \frac{2e^{3x} + 2e^{2x}(x+1)}{50(1+e^x)^3} + \frac{1}{5000}\\
        &=& \frac{2e^{2x}(2-e^x)}{(1+e^x)^4} + \frac{e^x\left(x+2 + e^x(2-x)\right)}{50(1+e^x)^3} + \frac{1}{5000}.
    \end{eqnarray*}
    Using simple computations one can check $\cH''(x)$ is negative for some $x\in\R$, e.g., one can check that $\cH''(3) < 0$. Therefore, function $\cH(x)$ is non-convex. Since convexity and smoothness of function $\cH(x)$ is equivalent to the cocoercivity of its gradient $\nabla \cH(x)$ \citep{nesterov2018lectures}, we conclude that Hamiltonian operator is non-cocoercive.
\end{proof}

Finally, one can use the optimization viewpoint of the Hamiltonian method and derive $\cO(\nicefrac{1}{K})$ random-iterate convergence guarantees in terms of $\|\nabla F(x^K)^\top F(x^K)\|^2$ when the Jacobian of $F$ is Lipschitz-continuous but $F$ is not necessary monotone. To show this the following lemma.

\begin{lemma}[Point-dependent smoothness of Hamiltonian function]\label{lem:point_dependent_smoothness_Hamiltonian}
    Let operator $F:\R^d \to \R^d$ be $L$-Lipschitz, its Jacobian $\nabla F(x)$ be $\Lambda$-Lipschitz, and $\cH(x) = \frac{1}{2}\|F(x)\|^2$. Then for any $x,y\in \R^d$ the following inequality holds:
    \begin{equation}
        \cH(y) \le \cH(x) + \langle\nabla \cH(x), y - x \rangle + \frac{L^2 + \Lambda\|F(x)\|}{2}\|y - x\|^2. \label{eq:point_dependent_smoothness_Hamiltonian}
    \end{equation}
\end{lemma}
\begin{proof}
    Since $F$ is $L$-Lipschitz its Jacobian has bounded norm: $\|\nabla F(x)\| = \|\nabla F(x)^\top\| \le L$ for all $x \in \R^d$. Using this and $\Lambda$-Lipschitzness of the Jacobian, we derive for any $x,y \in \R^d$
    \begin{eqnarray}
        \|\nabla \cH(y) - \nabla \cH(x)\| &=& \|\nabla F(y)^\top F(y) - \nabla F(x)^\top F(x)\|\notag\\
        &\le& \|\nabla F(y)^\top F(y) - \nabla F(y)^\top F(x)\| + \|\nabla F(y)^\top F(x) - \nabla F(x)^\top F(x)\|\notag\\
        &\le& \|\nabla F(y)^\top\|\cdot \|F(y) - F(x)\| + \|\nabla F(y)^\top - \nabla F(x)^\top\|\cdot \|F(x)\|\notag\\
        &\le& \|\nabla F(y)\|\cdot L\|x-y\| + \|\nabla F(y) - \nabla F(x)\|\cdot \|F(x)\|\notag\\
        &\le& L^2\|x - y\| + \Lambda\|F(x)\|\cdot \|x-y\|. \label{eq:Hamilton_lip_grad}
    \end{eqnarray}
    Next, following standard arguments \citep{nesterov2018lectures}, we get
    \begin{eqnarray*}
        \cH(y) &=& \cH(x) + \int\limits_{0}^1\langle \nabla \cH(x + t(y-x)), y-x \rangle dt\\
        &=& \cH(x) + \langle\nabla \cH(x), y - x \rangle + \int\limits_{0}^1\langle \nabla \cH(x + t(y-x)) - \nabla\cH(x), y-x \rangle dt\\
        &\le& \cH(x) + \langle\nabla \cH(x), y - x \rangle + \int\limits_{0}^1 \|\nabla \cH(x + t(y-x)) - \nabla\cH(x)\|\cdot \|y-x\| dt\\
        &\overset{\eqref{eq:Hamilton_lip_grad}}{\le}& \cH(x) + \langle\nabla \cH(x), y - x \rangle + \int\limits_{0}^1 (L^2 + \Lambda\|F(x)\|)t\|y-x\|^2 dt\\
        &=& \cH(x) + \langle\nabla \cH(x), y - x \rangle + \frac{L^2 + \Lambda\|F(x)\|}{2}\|y - x\|^2.
    \end{eqnarray*}
\end{proof}

\begin{theorem}[Best-iterate convergence of \eqref{DetHamiltonian}]\label{thm:best_iter_conv_HGM_appendix}
    Let operator $F:\R^d \to \R^d$ be $L$-Lipschitz, its Jacobian $\nabla F(x)$ be $\Lambda$-Lipschitz. Then for any $K\ge 0$ we have
    \begin{equation}
        \min\limits_{k=0,1,\ldots, K} \|\nabla F(x^k)^\top F(x^k)\|^2 \le \frac{\|F(x^0)\|^2}{\gamma\left(2 - \gamma(L^2 + \Lambda \|F(x^0)\|)\right)(K+1)}, \label{eq:best_iter_conv_HGM_appendix}
    \end{equation}
    where the sequence $x^0, x^1, \ldots, x^K$ is generated by \eqref{DetHamiltonian} with stepsize
    \begin{equation*}
        \gamma \le \frac{2}{L^2 + \Lambda\|F(x^0)\|}.
    \end{equation*}
    Moreover, for all $k \ge 0$ we have
    \begin{equation}
        \|F(x^{k+1})\| \le \|F(x^k)\|. \label{eq:relax_property_HGM}
    \end{equation}
\end{theorem}
\begin{proof}
    We start with applying Lemma~\ref{lem:point_dependent_smoothness_Hamiltonian}: taking $y = x^{k+1} = x^k - \gamma \nabla \cH(x^k)$ and $x = x^k$ in \eqref{eq:point_dependent_smoothness_Hamiltonian}, we get
    \begin{eqnarray}
        \cH(x^{k+1}) &\overset{\eqref{eq:point_dependent_smoothness_Hamiltonian}}{\le}& \cH(x^k) + \langle \cH(x^k), x^{k+1} - x^k\rangle + \frac{L^2 + \Lambda\|F(x^{k})\|}{2}\|x^{k+1} - x^k\|^2\notag\\
        &=& \cH(x^k) - \frac{\gamma}{2}\left(2 - \gamma(L^2 + \Lambda \|F(x^k)\|)\right)\|\nabla \cH(x^k)\|^2. \label{eq:best_iter_conv_HGM_tech1}
    \end{eqnarray}
    Using this inequality we will derive \eqref{eq:relax_property_HGM} by induction. For $k = 0$ we use our assumption on $\gamma$ and get that $2 - \gamma(L^2 + \Lambda \|F(x^k)\| \ge 0$. Therefore, the second term in the right-hand side of \eqref{eq:best_iter_conv_HGM_tech1} for $k=0$ is non-positive. This implies that $\cH(x^1) \le \cH(x^0)$, which is equivalent to $\|F(x^1)\| \le \|F(x^0)\|$. Next, assume that for some $K > 0$ inequality \eqref{eq:relax_property_HGM} holds for $k = 0,1,\ldots, K-1$. Let us derive that \eqref{eq:relax_property_HGM} holds for $k = K$ as well. Using \eqref{eq:best_iter_conv_HGM_tech1} and our inductive assumption, we derive
    \begin{eqnarray*}
        \cH(x^{K+1}) &\le& \cH(x^K) - \frac{\gamma}{2}\left(2 - \gamma(L^2 + \Lambda \|F(x^K)\|)\right)\|\nabla \cH(x^K)\|^2\\
        &\overset{\eqref{eq:relax_property_HGM}}{\le}& \cH(x^K) - \frac{\gamma}{2}\left(2 - \gamma(L^2 + \Lambda \|F(x^0)\|)\right)\|\nabla \cH(x^K)\|^2 \le \cH(x^K),
    \end{eqnarray*}
    where in the last inequality we use our assumption on $\gamma$. Therefore, $\|F(x^{K+1})\| \le \|F(x^K)\|$, i.e., \eqref{eq:relax_property_HGM} holds for all $k \ge 0$. Using this, we continue our derivation from \eqref{eq:best_iter_conv_HGM_tech1}:
    \begin{eqnarray*}
        \cH(x^{k+1}) &\le& \cH(x^k) - \frac{\gamma}{2}\left(2 - \gamma(L^2 + \Lambda \|F(x^k)\|)\right)\|\nabla \cH(x^k)\|^2 \\
        &\overset{\eqref{eq:relax_property_HGM}}{\le}& \cH(x^k) - \frac{\gamma}{2}\left(2 - \gamma(L^2 + \Lambda \|F(x^0)\|)\right)\|\nabla \cH(x^k)\|^2.
    \end{eqnarray*}
    Summing up the above inequality for $k = 0,1,\ldots, K$ and rearranging the terms, we obtain
    \begin{eqnarray*}
        \frac{1}{K+1}\sum\limits_{k=0}^K \|\nabla \cH(x^k)\|^2 &\le& \frac{2}{\gamma\left(2 - \gamma(L^2 + \Lambda \|F(x^0)\|)\right)(K+1)}\sum\limits_{k=0}^{K}\left(\cH(x^k) - \cH(x^{k+1})\right)\\
        &\le& \frac{2\cH(x^0)}{\gamma\left(2 - \gamma(L^2 + \Lambda \|F(x^0)\|)\right)(K+1)}.
    \end{eqnarray*}
    Finally, using the definition of $\cH$ and
    \begin{equation*}
        \min\limits_{k=0,1,\ldots, K} \|\nabla F(x^k)^\top F(x^k)\| \le \frac{1}{K+1}\sum\limits_{k=0}^K \|\nabla \cH(x^k)\|^2,
    \end{equation*}
    we get \eqref{eq:best_iter_conv_HGM_appendix}.
\end{proof}

\end{document}